\documentclass[11pt]{amsart}
\usepackage{latexsym}
\usepackage{amsxtra}
\usepackage{verbatim}
\usepackage{color}
\usepackage{amsthm,amsmath,amsfonts,amssymb,mathrsfs,amscd,graphics}
\usepackage{amssymb}
\usepackage{appendix}
\usepackage{ifpdf}
\usepackage{enumerate}
\usepackage{fancyhdr}
\usepackage{mathrsfs}
\usepackage[mathscr]{eucal}
\usepackage{upgreek}
 \usepackage{wasysym}
\usepackage{esint}
\usepackage{pstricks}
\usepackage{color}
\usepackage{cancel}
\usepackage{verbatim}
\usepackage[all]{xy}

\usepackage[margin=1.3in]{geometry}

\usepackage[colorlinks=true]{hyperref}

\usepackage{color}

\usepackage{lipsum}
\usepackage{float}
\usepackage{mathtools}

\usepackage{longtable}
\allowdisplaybreaks
\usepackage{tikz}
\usepackage{tikz-cd}
\usepackage{amstext, amsbsy}
\usepackage{amsfonts, extarrows}
\newtheorem{theorem}{Theorem}[section]
\newtheorem{lemma}[theorem]{Lemma}
\theoremstyle{definition}

\newtheorem{proposition}[theorem]{Proposition}
\newtheorem{corollary}[theorem]{Corollary}
\newtheorem{conjecture}[theorem]{Conjecture}

\theoremstyle{remark}
\newtheorem{remark}[theorem]{Remark}
\numberwithin{equation}{section}

   \newcommand{\cM}{{\mathcal M}}

    \newcommand{\cH}{{\mathcal H }}

  \newcommand{\C}{{\mathbb C}}
  
  \newcommand{\Z}{{\mathbb Z}}
  
  \newcommand{\N}{{\mathbb  N}}

\newcommand{\cL}{\mathcal L}

\newcommand{\bt}{\mathbf{t}}

\newcommand{\LL}{\llangle[\big]}
\newcommand{\RR}{\rrangle[\big]}

\makeatletter
\newsavebox{\@brx}
\newcommand{\llangle}[1][]{\savebox{\@brx}{\(\m@th{#1\langle}\)}%
  \mathopen{\copy\@brx\kern-0.5\wd\@brx\usebox{\@brx}}}
\newcommand{\rrangle}[1][]{\savebox{\@brx}{\(\m@th{#1\rangle}\)}%
  \mathclose{\copy\@brx\kern-0.5\wd\@brx\usebox{\@brx}}}
\makeatother

\newcommand{\<}{\langle}
\renewcommand{\>}{\rangle}

\usepackage{booktabs}

\usepackage{tikz}
\usepackage[framemethod=tikz]{mdframed}
\usetikzlibrary{shapes,decorations}
\usetikzlibrary{arrows,matrix,positioning}
\usepackage[most]{tcolorbox}
\usepackage{color}
\usepackage{amsfonts,mathrsfs}

\newtheorem{axiom}[theorem]{Axiom}

\newcommand{\fj}[2]{#1 | #2 \rangle}

\renewcommand{\arraystretch}{1.1}
\begin{document}
% \title[short text for running head]{full title}
\title[FJRW theory of two-variable polynomials]{Semisimple FJRW theory of polynomials \\ with two variables}

%\date{\today}
%    Only \author and \address are required; other information is
%    optional.  Remove any unused author tags.
%    author one information
% \author[short version for running head]{name for top of paper}

\author{Amanda Francis, Weiqiang He, and Yefeng Shen}

\maketitle

\begin{abstract}
We study the Dubrovin-Frobenius manifold in the Fan-Jarvis-Ruan-Witten theory of Landau-Ginzburg pairs $(W, \<J\>)$, where $W$ is an invertible nondegenerate quasihomogeneous polynomial with two variables and $\<J\>$ is the minimal admissible group of $W$. 
We conjecture that the Dubrovin-Frobenius manifolds from these FJRW theory are semisimple. 
We show the conjecture holds true for simple singularities and almost all Brieskorn-Pham polynomials.
For Brieskorn-Pham polynomials, the result follows from the calculation of a quantum Euler class in the FJRW theory. %and the quantum multiplication of this class along a special direction. 
As a consequence, our result shows that for the FJRW theory of these Landau-Ginzburg pairs, both a Dubrovin type conjecture and a Virasoro conjecture hold true. % which is proposed by the last two authors.
%On the other hand, we construct a full exceptional collection of objects in the equivariant matrix factorizations for these singularities. 
Along the way we give explicit formulae for computing concave genus-$0$ $5$-point FJRW invariants (correlators) corresponding to admissible Landau-Ginzburg pairs $(W,G)$.
\end{abstract}

{\hypersetup{linkcolor=black}\setcounter{tocdepth}{1} \tableofcontents}

\section{Introduction}

\subsection{Dubrovin conjecture for semisimple quantum cohomology}

Let $X$ be a Fano variety and let $n$ be the dimension of its cohomology. 
In Dubrovin's 1998 ICM talk \cite{Dubrovin}, he conjectured that the bounded derived category of coherent sheaves of $X$ admits a full system of $n$ exceptional objects (in the sense of \cite{BP}) if and only if the quantum cohomology of $X$ is semisimple.
This is the first of three conjectures in \cite{Dubrovin} which unveils remarkable relationships between the algebraic structure of the Fano manifold (from the derived category) and its analytic structure (from the quantum cohomology).

The quantum cohomology is said to be semisimple if the underlying connected {\em Dubrovin-Frobenius manifold} (invented by Dubrovin) has a point with a semisimple Frobenius algebra. %where the quantum product is given by the genus-zero Gromov-Witten invariants of the variety.
Dubrovin's conjecture has been generalized by Bayer-Manin \cite{BM} and the Fano condition can be dropped \cite{Bayer}.
Other examples of semisimple Dubrovin-Frobenius manifolds have been found in Saito's theory of primitive forms \cite{Sai} and other places.

Semisimple Dubrovin-Frobenius manifolds have remarkable properties. 
Givental reconstructed higher genus formulas \cite{G-ss} for semisimple Dubrovin-Frobenius manifolds, and proved that the higher genus formula satisfies the Virasoro constraints \cite{G-vir}.
The uniqueness of Givental's formula has been proved by Teleman \cite{Teleman} for semisimple Dubrovin-Frobenius manifold that satisfies a homogeneity condition.

%\subsubsection{Results on semisimplicity}
%Results on semisimplicity in quantum cohomology including earlier result for toric fano \cite{??} and some recent result \cite{Abr, Ke}.
%\subsubsection{Application of semisimplicity}
%Application to higher genus \cite{G-ss, Teleman}  and Virasoro constraints .

%\subsubsection{Status of Dubrovin's conjecture} Dubrovin conjecture has been proved or investigated for many examples. In QC, \cite{??}, In singularity theory and derived category of matrix factorizations \cite{AT, Kra}.

%\newpage

\subsection{FJRW theory and semisimplicity}
The focus of this paper will be the Dubrovin-Frobenius manifold that appears in Landau-Ginzburg (LG) A-model of a pair $(W, G)$, where $W: \C^N\to \C$ is a quasi-homogeneous polynomial and $G$ is an abelian symmetry group of $W$. 
Here a polynomial $W$ is called {\em quasihomogeneous} if there exists positive integers $(d; w_1, \ldots, w_N)$ such that
$$\lambda^d W(x_1,\ldots,x_N) = W(\lambda^{w_1}x_1,\ldots,\lambda^{w_N} x_N), \forall \lambda\in \C.$$
The rational number $q_i:={w_i\over d}$ is called the weight of the variable $x_i$. 

We will consider the so-called {\em admissible LG pair $(W, G)$}, where the quasihomogeneous polynomial $W$ is required to be {\em nondegenerate}, that is, (1) it has only isolated critical points at the origin, (2) it contains no $x_ix_j$ term for $i\neq j$.
%the choice of weights of variables are unique.
The group $G$ is required to be a subgroup of the {\em maximal symmetry group of $W$}
$$G_W:=\{(\lambda_1, \ldots, \lambda_N)\in \C^*\mid W(\lambda_1x_1, \ldots, \lambda_Nx_N)=W(x_1, \ldots, x_N)\},$$
that contains the {\em exponential grading element} 
$$J:=(e^{2\pi\sqrt{-1}q_1}, \ldots, e^{2\pi\sqrt{-1}q_N})\in \C^N.$$

For any admissible LG pair $(W, G)$, Fan, Jarvis, and Ruan constructed an intersection theory on the moduli space of W-spin structures of G-orbifold curves in %a series of works 
\cite{FJR1, FJR}, based on a proposal of Witten. 
This theory is called Fan-Jarvis-Ruan-Witten (FJRW) theory nowadays. It is analogous to Gromov-Witten (GW) theory and generalizes the theory of $r$-spin curves, where $W = x^r$ and $G = \<J\>$.
The main ingredient of the FJRW theory is a cohomological field theory (CohFT) 
$$\{\Lambda^{W, G}_{g, k}: \cH_{W, G}^{\otimes k}\to H^*(\overline{\cM}_{g,k})\}_{2g-2+k>0}$$
on a state space $\cH_{W, G}$. The state space is the $G$-invariant subspace of the middle-dimensional relative cohomology for $W$ with a nondegenerate pairing $\< \cdot, \cdot\>$; see \cite[Section 3.2]{FJR}. Integrating the CohFT over the moduli spaces of stable curves with markings produces the so-called {\em FJRW invariants}.
We refer the readers to \cite{FJR} for more details on FJRW invariants.

Despite their different geometric nature, there are remarkable connections between FJRW invariants and GW invariants for Calabi-Yau manifolds via Landau-Ginzburg/Calabi-Yau correspondence, and connections between FJRW theory and Saito's theory of primitive forms via LG mirror symmetry.  

\subsubsection{A Dubrovin conjecture in FJRW theory}
According to \cite[Corollary 4.2.8]{FJR}, the genus-zero FJRW theory defines a formal Dubrovin-Frobenius manifold on the state space $\cH_{W, G}$. 
%In general, for any admissible pair $(W, G)$ and 
For any $\bt\in \cH_{W, G}$, 
there is a (formal) Frobenius algebra 
$(\cH_{W, G}, \star_{\bt}, \< \cdot, \cdot\>)$.
We say the FJRW theory of $(W, G)$ or the Dubrovin-Frobenius manifold on $\cH_{W, G}$ is semisimple if the corresponding Frobenius algebra $(\cH_{W, G}, \star_{\bt}, \< \cdot, \cdot \>)$ is semisimple at some point $\bt\in \cH_{W, G}$.
%We say the FJRW theory of $(W, G)$ is semisimple if the underlying Dubrovin-Frobenius manifold is semisimple.
Here is a folklore conjecture of Dubrovin type. 
\begin{conjecture}
[Dubrovin conjecture in FJRW theory]
\label{dubrovin-fjrw}
Let $n$ be the rank of the state space $\cH_{W, G}$.
The derived category of $G$-equivariant matrix factorizations of $W$ admits a full system of $n$ exceptional objects if and only if the FJRW theory of $(W, G)$ is semisimple.
\end{conjecture}
%For the $r$-spin case, the conjecture holds. This is well known. % when $(W=x^r, G = \<J\>)$. 
%There is full system of exceptional objects $\{\C^{\rm st}, \C(-1)^{\rm st}, \ldots, \C(2-r)^{\rm st}\}$, where $\C^{\rm st}$ is a particular {\em Koszul matrix factorization} called the {\em stablization of the residue field $\C$} in \cite{Dyc} and $\C(j)$ is a twist of $\C$.

 \subsubsection{Invertible polynomials with maximal admissible group}
It is natural to ask for which LG pairs is the FJRW theory  semisimple.
A large class of semisimple Dubrovin-Frobenius manifolds has been found in FJRW theory of the so-called {\em invertible polynomials} with the maximal symmetry group.
We say a polynomial $W: \C^N\to C$ is invertible if after rescaling and permutation of the variables, %it can be written as 
$$W=\sum\limits_{i=1}^{N}\prod\limits_{j=1}^Nx_i^{a_{i,j}}$$ and the $N\times N$ {\em exponential matrix} $E_W:=(a_{i,j})$ is an invertible matrix.
According to the classification in \cite{KS}, up to permutation of variables, an invertible polynomial $W$ must be a Thom--Sebastiani sum of polynomials of the following types
\begin{enumerate}
\item Fermat: $x_1^{a}$; 
\item chain: $x_1^{a_1}+x_1x_2^{a_2}+\ldots +x_{m-1}x_m^{a_m}$; 
\item loop: $x_1^{a_1}x_2+x_1x_2^{a_2}+\ldots +x_m^{a_m}x_1$.
\end{enumerate} 
The Thom--Sebastiani sums of polynomials of Fermat type are also called of Brieskorn--Pham type. %In this paper, the Brieskorn-Pham type polynomials with two variables $x_1^{a_1}+x_2^{a_2}$ will be the main focus.

If the invertible polynomial $W$ has no weight-${1\over 2}$ variable, then mirror symmetry between the FJRW theory of the admissible pair $(W, G_W)$ and a Landau-Ginzburg B-model of the mirror polynomial (called {\em Berglund-H\"ubsch mirror} or {\em Berglund-H\"ubsch-Krawitz mirror})
$$W^T=\sum\limits_{i=1}^{N}\prod\limits_{j=1}^Nx_i^{a_{j,i}}$$ has been studied in \cite[Theorem 1.2]{HLSW} at all genus. 
%In particular, the underlying Frobenius algebra between $\cH_{W, G_W}$ and the Jacobian algebra of $W^T$ is proved \cite{Krawitz}.
As a consequence, the Dubrovin-Frobenius manifold in the A-model is equivalent to the Dubrovin-Frobenius manifold in the B-model. 
The latter are constructed via Saito's theory of primitive forms \cite{Sai}. 
They are semisimple as the miniversal deformations of $W^T$ are generically Morse functions.
%Besides the result in \cite{HLSW}, 

On the other hand, if the admissible group $G$ satisfies $G\neq G_W$, not much is known on the semisimplicity of the FJRW theory of $(W, G)$. 

\begin{remark}
If $(W, G_W)$ is an admissible LG pair and $W$ is invertible, then Dubrovin type Conjecture \ref{dubrovin-fjrw} for $(W, G_W)$ follows from the construction of a full system of exceptional objects by Favero, Kaplan, and Kelly in \cite{FKK} and the mirror theorem in \cite{HLSW}.
If $N=2$, the result was proved earlier by Habermann and Smith \cite{HaSm} and Habermann \cite{Habermann-hms}.
The result for invertible polynomials of chain type was proved earlier by Hirano and Ouchi \cite{HO}. 
%of a system of full exceptional objects for the derived category of $G_W$-equivariant matrix factorizations of $W$ in \cite{AT} and the mirror theorem in \cite{HLSW}.
%They also computed the Stokes matrix.
%For invertible polynomials with at most three variables, Kravets \cite{Kravets} has constructed full strongly exceptional collections.
\end{remark}

\subsection{Main result}
\label{sec-main-result}
In this paper, we consider another class of admissible pairs $(W, G)$ where $W$ is an invertible polynomial with two variables. % and $\<J\>$ is the minimal admissible group that is generated by the exponential grading element $J$.
Up to permutation of variables, all such polynomials are listed in Table \ref{table-basis}.
For simplicity, we always assume $\gcd(w_1, w_2)=1$. 
Thus the cyclic group $\<J\>$ has order $d$.
In particular, $\<J\>= G_W$ if and only if $\delta=1.$ 
Other notions in Table \ref{table-basis} will be explained in Section \ref{invertible-two}.
We conjecture that
\begin{conjecture}
\label{main-conj}
The FJRW theory of any admissible LG pair $(W, G)$ is semisimple if $W$ is an invertible polynomial with at most two variables.  %generically semisimple near the origin of $\cH_{W,  \<J\>}$.
\end{conjecture}

\begin{table}
\centering
\caption{Degree $d$ invertible polynomials with two variables}
\label{table-basis}
 \resizebox{\textwidth}{!}{
\renewcommand{\arraystretch}{1.8}
\begin{tabular}{|c|c|c|c|}\hline
Type & Brieskorn-Pham & Chain & Loop\\\hline
$W$&$x_1^{a_1}+x_2^{a_2}$&$x_1^{a_1}+x_1x_2^{a_2}$&$x_1^{a_1}x_2+x_1x_2^{a_2}$\\\hline
$E_{W}$&$\begin{bmatrix}a_1 &\\& a_2\end{bmatrix}$&
$\begin{bmatrix}a_1&\\1& a_2\end{bmatrix}$&$\begin{bmatrix} a_1 & 1 \\1 & a_2 \end{bmatrix}$\\\hline
$\delta$ & $\gcd(a_1, a_2)$ & $\gcd(a_1-1, a_2)$ &  $\gcd(a_1-1, a_2-1)$\\\hline
New form & $x_1^{\delta w_2}+x_2^{\delta w_1}$ & $x_1^{\delta w_2+1}+x_1x_2^{\delta w_1}$ & $x_1^{\delta w_2+1}x_2+x_1x_2^{\delta w_1+1}$\\\hline
$d$ & $\delta w_1 w_2$ & $\delta w_1 w_2 +w_1$ & $\delta w_1 w_2+w_1 +w_2$\\\hline
$\mu_{nar}$ & $\delta w_1 w_2-w_1-w_2+1$ & $\delta w_1 w_2$ & $\delta w_1 w_2+w_1+w_2-1$\\ \hline
$\mu_{bro}$ & $\delta-1$ & $\delta$ & $\delta+1$ \\ \hline
%$\mu_{inv}$ & $\delta w_1 w_2-w_1-w_2+1$ & $\delta w_1 w_2+1$ & $\delta w_1 w_2+w_1+w_2+1$\\ \hline
$\mu$ & $\delta w_1 w_2+\delta-w_1-w_2$ & $\delta w_1 w_2+\delta$ & $\delta w_1 w_2+\delta+w_1+w_2$ \\ \hline
%$\mu-\mu_{inv}$ & $\delta-1$ & $\delta-1$ & $\delta-1$ \\ \hline
\end{tabular}
}
\end{table}

\iffalse

\subsubsection{An analogy to quantum cohomology}

Before we state the main result, let us explain briefly why we are interested in these admissible pairs by answering the following two questions: 
\begin{enumerate} 
\item why the minimal group $\<J\>$? 
\item why polynomials with two variables?
\end{enumerate}
Our starting point is based on two analogy in the study of quantum cohomology.
\begin{itemize}
\item Choice of $\<J\>$ smooth manifolds vesus orbifolds.
\item Why two variables, an analogy to quantum cohomology,
Quadrics vesus two-variables.
Kapranov. \cite{Kap}
degree-variable comparison.
\end{itemize}

\fi

\iffalse
We consider invertible polynomials of two-variables. 
According to the classification result in \cite{KS}, up to permutation of variables, an invertible polynomials of two-variables $W$ must be one of the following forms:
\begin{enumerate}
\item Brieskorn-Pham type: $x_1^{a_1}+x_2^{a_2}$;
\item Chain type: $x_1^{a_1}+x_1x_2^{a_2}$;
\item Loop type: $x_1^{a_1}x_2+x_1x_2^{a_2}$.
\end{enumerate}
\fi

%If $N=1,$ the result is easy. %follows from mirror symmetry. 

We recall that the {\em central charge} of the polynomial $W$ is given by 
$$\widehat{c}_W:=\sum\limits_{i=1}^N(1-2q_i).$$
The main result of the paper is %are the following two statements.
\begin{theorem}
\label{main}
Conjecture~\ref{main-conj} holds for $(W, G)$ if $\widehat{c}_W<1$ or $W$ is a Brieskorn-Pham polynomial and $G=\<J\>$. % that is not $x^3+y^{15}$ or $x^5+y^{20}$.
\end{theorem}

%\newpage

%We will discuss in two cases, (1) $\widehat{c}_W<1$, (2) $\widehat{c}_W\geq 1$.
 %based on the choices of $(d; w_1, w_2)$.
%\begin{enumerate}
%\item when $\widehat{c}_W<1$, that is, $W=x_1^2+x_2^m$ for $m\geq 2$ or $W=x_1^3+x_2^3, x_1^3+x_2^4, x_1^3+x_2^5$. % contains at least one weight-${1\over 2}$ variable.
%\item when $\widehat{c}_W\geq 1$.
%\end{enumerate}
The result for the first case can be essentially obtained from the result for FJRW theory of simple singularities in \cite{FJR}, %by mirror symmetry. 
%We will list them along with some other examples in Section \ref{special-cases}.
see Theorem \ref{theorem-ade} and Theorem \ref{theorem-simple} in Section \ref{special-cases}.
The main focus is the second case.

According to \cite[Theorem 3.4]{Abr}, a Frobenius algebra is semisimple if and only if the {\em characteristic element} of the Frobenius algebra is a unit. 
In quantum cohomology or FJRW theory, such a  characteristic element  is also called a {\em quantum Euler class}, see \eqref{quantum-euler-j2}. 
We will complete the proof of Theorem \ref{main} by computing the quantum Euler classes explicitly in Section \ref{sec-nonmax}.
The key step is actually computing a specific genus-$0$ $5$-point FJRW invariant in \eqref{definition-5-point} which satisfies the {\em concavity axiom} in~\cite[Theorem 4.1.8]{FJR}. 
We give a general method to compute such FJRW invariants in Section \ref{5-point-appendix}. 
%The current method does not work for the two exceptional cases $x^3+y^{15}$ or $x^5+y^{20}$, see Lemma \ref{delta-mu}.

%When we have $\<J\>=G_W$?
\begin{remark}
%We remark that f
\begin{enumerate}
\item
The same method we used here in the proof of Theorem \ref{main} also works for some other examples, such as the chain polynomial $x_1^3 +x_1x_2^8$, see \cite[Proposition 3.8]{HS}. 
However, it involves computation of certain correlators with broad insertions, which is not known in general.

\item
For invertible polynomials with more than two variables, things get more complicated. 
When $N=3$, the FJRW theory of $(W=x_1^3+x_2^3+x_3^3, \<J\>)$ studied in \cite{LSZ} is not semisimple while the FJRW theory of $(W=x_1^p+x_2^q+x_3^r, \<J\>)$ is semisimple if any two of $p, q, r$ are coprime (because in this case  $\<J\>=G_W$). 
\end{enumerate}
%That is because we have $\<J\>=G_W$ for the later case.
%On the other hand, the FJRW theory of $(W=x_1^3+x_2^3+x_3^3, G)$ for any larger group $G\neq \<J\>$ is aslo semisimple, as the FJRW theory is equivalent to the FJRW theory of $(W=x_1^3+x_2^3+x_3^3, G_W).$
%In general, we don't know for which $G$, the FJRW theory of $(W, G)$ is semisimple.
\end{remark}

\subsection{Applications of Theorem \ref{main}}
Now we discuss some consequences of Theorem~\ref{main}.

\subsubsection{Dubrovin conjecture in FJRW theory}
Habermann has proved that if $W$ is an invertible polynomial with two variables, then for any admissible group $G$, the corresponding derived category of matrix factorizations ${\rm mf}(\C^2, G, W)$ has a tilting object of length $\mu$, see \cite[Corollary 1]{Habermann-curve}. This result implies that the category ${\rm mf}(\C^2, G, W)$ admits a full system of $\mu$ exceptional objects. Combining this result with Theorem~\ref{main}, we immediately obtain
\begin{corollary}
Dubrovin Conjecture~\ref{dubrovin-fjrw} holds for $(W, G)$ if $\widehat{c}_W<1$ or $W$ is a Brieskorn-Pham polynomial and $G=\<J\>$.
%\begin{enumerate}
%\item $\widehat{c}_W<1$ or
%\item $G=\<J\>$ and $W$ is a Brieskorn-Pham polynomial. % that is not $x^3+y^{15}$ or $x^5+y^{20}$.
%\end{enumerate}
\end{corollary}

%\subsection{Invertible polynomials of two variables}

\subsubsection{Virasoro constraints}
%Now we consider the simplest cases when $W$ is a Fermat homogeneous polynomial.
%Semisimple Frobenius manifolds also exists for admissible LG A-model pairs even if $G\neq G_W$, such as pairs
%$$(W=x_1^d+x_2^d, \quad G=\<J_W\>), \quad d\geq 2.$$

%We see part (2) is Theorem \ref{general-pillow}.
Applying Theorem \ref{main} and Givental Theorem \cite[Theorem 7.7]{G-vir}, 
we prove a Virasoro conjecture for the corresponding FJRW theory, proposed in \cite{HS}.
\begin{corollary}
Let $W$ be an invertible polynomial with at most two variables. 
The Virasoro conjecture \cite[Conjecture 0.6]{HS} holds true for the FJRW theory of $(W, G)$  if $\widehat{c}_W<1$ or $W$ is a Brieskorn-Pham polynomial and $G=\<J\>$.  
%\begin{enumerate}
%\item $\widehat{c}_W<1$ or
%\item $G=\<J\>$ and $W$ is a Brieskorn-Pham polynomial. % that is not $x^3+y^{15}$ or $x^5+y^{20}$.
%\end{enumerate}
% if either $\widehat{c}_W<1$, or $G=\<J\>$ and $W$ is a Fermat polynomial that is not $x^3+y^{15}$ or $x^5+y^{20}$.
\end{corollary}
We refer the readers to \cite[Section 3]{HS} for more details about the Virasoro conjecture for semisimple FJRW theories and the proof.

\subsection*{Acknowledgement}
We would like to thank Hua-zhong Ke for pointing out the references \cite{Abr, Ke} and Matthew Habermann for pointing out the reference \cite{Habermann-curve}.
%After the first draft of the paper was posted on Arxiv, we were noticed by Matthew Habermann that for any invertible polynomial $W$ of two-variables with the minimal admissible group $\<J\>$, a full exceptional collection of objects was constructed for the category $MF_{\<J\>}(W)$ in \cite{Habermann-curve}. We thank Matthew Habermann for the communications.
W.H. is partially supported by National Key Research and Development Program of China (2023YFA1009802), NSFC grant 12422104.
Y.S. is partially supported by a Simons Collaboration Grant.

%\newpage

\section{FJRW theory of invertible polynomials with two variables}
\label{invertible-two}

\subsection{Invertible polynomials with two variables}

From now on, we will consider admissible LG pair $(W, \<J\>)$, where $W: \C^2\to \C$ is an invertible polynomial with two variables that is listed in Table \ref{table-basis} and $\<J\>$ is the order $d$ cyclic group generated by the element
$J=(e^{2\pi\sqrt{-1}q_1}, e^{2\pi\sqrt{-1}q_2})$.
Recall that  $w_1, w_2$ are coprime positive integers and $(q_1, q_2)=\left({w_1\over d}, {w_2\over d}\right)$. % are the weights of variables.
We call the triple $(d; w_1, w_2)$ the {\em weight system} of the LG pair $(W, \<J\>)$.
We can rewrite the three types of two-variable invertible polynomials as follows:
\begin{enumerate}
\item Brieskorn-Pham type: $W=x_1^{\delta w_2}+x_2^{\delta w_1}$, where $\delta=\gcd(a_1, a_2)$;
\item Chain type: $W=x_1^{\delta w_2+1}+x_1x_2^{\delta w_1}$,  where $\delta=\gcd(a_1-1, a_2)$;
\item Loop type:  $W=x_1^{\delta w_2+1}x_2+x_1x_2^{\delta w_1+1}$, where $\delta=\gcd(a_1-1, a_2-1)$.
\end{enumerate}
The {\em central charge} of the singularity $W$ is given by 
\begin{equation}
\label{central-charge}
\widehat{c}_W=2- {2w_1\over d}-{2w_2\over d}.
\end{equation}

\subsubsection{State space}
According to \cite[Formula (74)]{FJR}, there is a (graded) vector space isomorphism for the the state space of the FJRW theory of $(W, \<J\>)$
\begin{equation}
\label{state-space-iso}
\cH_{W, \<J\>}\cong\bigoplus_{g\in \<J\>} \cH_g, \quad \cH_g:=\left({\rm Jac}(W_g)\cdot\Omega_g\right)^{\<J\>}.
\end{equation}
Here $\Omega_g$ is the standard top form on the $g$-fixed locus ${\rm Fix}(g)\subset \C^2$, 
$W_g$ is the restriction of $W$ on ${\rm Fix}(g)$, and ${\rm Jac}(W_g)$ is the Jacobian algebra of $W_g$.

When ${\rm Fix}(g)=\{0\}\subset\C^2$, the space $\left({\rm Jac}(W_g)\cdot\Omega_g\right)^{\<J\>}$ is considered to be one-dimensional. It is spanned by a standard vector which we denote by $1 \vert g\>.$
We can write elements $g\in\<J\>$ as $g=J^m$ for some $m$ in the {\em set of narrow indices} 
\begin{equation}
\label{narrow-index}
{\bf Nar}=\{k\in \Z\mid 1\leq k\leq d-1, d\nmid kw_1, d\nmid kw_2\}.
\end{equation}

In general, using the isomorphism~\eqref{state-space-iso}, we will write the homogeneous elements of $\cH_{W, \<J\>}$ in the form of 
\begin{equation}
\label{element-general-form}
\gamma=f\cdot \Omega_g \vert g\>, \quad g\in \<J\>,
\end{equation}
where $f\in {\rm Jac}(W_g)$ and $f\cdot \Omega_g$ is $\<J\>$-invariant.
If ${\rm Fix}(g)=\{0\}$ we will call such an element $\gamma$ in~\eqref{element-general-form} a {\em broad element}. In fact, when $W$ is an invertible polynomial with two variables, it is easy to check that all the broad elements in $\cH_{W, \<J\>}$ are from the case $g=J^0={\rm id}\in \<J\>$. 
\begin{proposition}
If $W$ is an invertible polynomial of two variables, the vector space in~\eqref{state-space-iso} has a decomposition 
\begin{equation}\label{narrow-broad}
\cH_{W, \<J\>}=\cH_{\rm nar}\oplus \cH_{\rm bro}.
\end{equation}
Here $\cH_{\rm nar}$ is the subspace of narrow elements with a basis 
\begin{equation}
\label{narrow-generator}
\{\alpha_m:=1\vert J^{m}\> \mid m\in {\bf Nar}\},
\end{equation}
and $\cH_{\rm bro}$ is the subspace of broad elements with a basis $\{\beta_k\}$, where 
\begin{equation}
\label{broad-generator}
\beta_k:=
\begin{dcases}
x_1^{k w_2-1}x_2^{(\delta-k) w_1-1}dx_1\wedge dx_2, & 1\leq k \leq \delta-1, \text{if W is Brieskorn-Pham};\\
x_1^{k w_2}x_2^{(\delta-k) w_1-1}dx_1\wedge dx_2, & 0\leq k \leq \delta-1, \text{if W is chain};\\
x_1^{k w_2}x_2^{(\delta-k) w_1}dx_1\wedge dx_2, & 0\leq k \leq \delta, \text{if W is loop}.
\end{dcases}
\end{equation}
\end{proposition}
We denote the rank of the vector spaces  $\cH_{W, \<J\>}, \cH_{\rm nar}, \cH_{\rm bro}$  by $\mu, \mu_{\rm nar}, \mu_{\rm bro}$ respectively.
By definition $\mu=\mu_{\rm nar}+\mu_{\rm bro}.$
We denote the $G_W$-invariant subspace by $\cH_{W, \<J\>}^{G_W}$. 
\begin{lemma}\label{lemma-invariant-subspace}
For invertible polynomial $W: \C^2\to \C$, the subspace $\cH_{W, \<J\>}^{G_W}$ is spanned by all the narrow elements, $\beta_0$ and $\beta_\delta$ (if exists).
\end{lemma}

\subsubsection{Degree formula}
We write each element in $\<J\>$ as 
\begin{equation}
\label{exponent-rational}
g:=\left(\exp(2\pi\sqrt{-1}\theta_1(g)), \exp(2\pi\sqrt{-1}\theta_2(g))\right)\end{equation}
for a unique pair of rational numbers $(\theta_1(g), \theta_2(g))$ such that $\theta_1(g), \theta_2(g) \in [0, 1)$.
Let $$\iota_g:=\theta_1(g)+\theta_2(g)-q_1-q_2$$ be the degree shifting number of $g$.
Let $N_g$ be the complex dimension of the subspace ${\rm Fix}(g)$.
According to \cite[formula (47)]{FJR}, the complex degree of each homogeneous element $\phi\in \cH_{g}$ is given by 
\begin{equation}
\deg_\C\phi={N_g\over 2}+\iota_g.
\end{equation}
We can check directly that 
\begin{lemma}
\label{degree-lemma}
%Let $W$ be an invertible polynomial with two variables. 
Let $W:\C^2\to \C$ be an invertible nondegenerate quasihomogeneous polynomial, then for any homogeneous element $\phi\in \cH_{W, \<J\>}$, we have 
$$0\leq \deg_\C\phi\leq \widehat{c}_W.$$ 
The equality holds only if $\phi=\alpha_1$ or $\alpha_{d-1}.$ Furthermore, we have 
\begin{equation}
\label{degree-homogeneous}
\deg_\C\alpha_1=0; \quad \deg_\C\alpha_{d-1}=\widehat{c}_W ; \quad \deg_\C\beta_k={\widehat{c}_W\over 2}.
\end{equation}
%If  $W$ has no weight ${1\over 2}$ variable, then $\alpha_2\in \cH_{W, \<J\>}$ and $$\deg_\C\alpha_2= {w_1+w_2\over d}.$$
\end{lemma}

\subsection{Dubrovin-Frobenius manifolds in FJRW theory}
% and some computation tools
We briefly review the  quantum products and Dubrovin-Frobenius manifolds in FJRW theory \cite{FJR}. 

Following \cite[Definition 2.2.7]{FJR}, for an admissible LG pair $(W, G)$, we denote the moduli space of stable $W$-structure by $\overline{\mathcal{W}}_{g, k}(g_1, \ldots, g_k)$ where $2g-2+k>0$. 
Roughly speaking, let $C$ be a genus-$g$ orbicurve with $k$-markings, decorated by the group elements $g_1, \ldots, g_k\in G$, 
a $W$-structure on the orbicurve $C$ \cite[Definition 2.2.1]{FJR} consists of $N$ orbifold line bundles $\cL_1, \ldots, \cL_N$ over $C$, satisfying certain constraint given by $W$ \cite[Definition 2.1.11]{FJR}. 
Virtual cycles for the moduli of $W$-structures are constructed in \cite{FJR1}, see \cite[Definition 4.1.6]{FJR}. 
The main output of FJRW theory is a cohomological field theory $\{\Lambda_{g,k}^{W, G}:  \cH_{W, G}^{\otimes k}\to H^*(\overline{\cM}_{g,k})\}$ \cite[Definition 4.2.1]{FJR}.
Let $\gamma_i\in \cH_{g_i}$, the {\em genus-$g$ $k$-point FJRW invariant (or correlator)} $\<\gamma_1, \ldots, \gamma_k\>_{g,k}$ is defined by
$$\<\gamma_1, \ldots, \gamma_k\>_{g,k}:=\int_{\overline{\mathcal{M}}_{g,k}}\Lambda^{W, G}_{g, k}(\gamma_1, \ldots, \gamma_k).$$
 
Let $\{\phi_i\}$ be a homogeneous basis of $\cH_{W, G}$ and $\{\phi^i\}$ be its dual basis. 
We denote the coordinate of the basis element $\phi_i$ by $t_i$. 
Let ${\bf t}:=\sum_{i} t_i \phi_i.$
According to  \cite[Corollary 4.2.8]{FJR}, the genus-zero FJRW theory defines a formal Dubrovin-Frobenius manifold on the state space $\cH_{W, G}$. 
The potential of the formal Dubrovin-Frobenius manifold of $\cH_{W, G}$ will be denoted by $F_0(\bt)$ here. 
Its third derivative is given by the double bracket notation
$${\partial^3\over\partial t_k\partial t_j\partial t_i} F_0(\bt)=\LL\phi_i,\phi_j, \phi_k\RR_{0,3}(\bt):=\sum_{n\geq 0}{1\over n!}\<\phi_i, \phi_j, \phi_k, \bt, \ldots, \bt\>_{0, n+3}.$$
The quantum product $\star_{\bt}$ in FJRW theory is defined by 
\begin{equation}
\label{quantum-product}
\<\phi_i\star_{\bt}\phi_j, \phi_k\>=\LL\phi_i,\phi_j, \phi_k\RR_{0,3}(\bt).
\end{equation}
The associativity of the quantum product $\star_{\bt}$ is equivalent to the Witten-Dijkgraaf-Verlinde-Verlinde (WDVV) equations in FJRW theory.
That is, for any elements $\gamma_1, \gamma_2, \gamma_3, \gamma_4\in \cH_{W, G}$, we have 
\begin{equation}
\label{wdvv-equation}
\sum_{a}\LL\gamma_1, \gamma_2, \phi_a\RR_{0,3}(\bt)\cdot\LL \phi^a, \gamma_3, \gamma_4\RR_{0,3}(\bt)
=
\sum_{a}\LL\gamma_1, \gamma_3, \phi_a\RR_{0,3}(\bt)\cdot\LL \phi^a, \gamma_2, \gamma_4\RR_{0,3}(\bt).
\end{equation}

Now we collect some properties of FJRW invariants, which will be used frequently later.  
\subsubsection{Concavity axiom} % Recall Concavity axiom 
%\textcolor{red}{State the concavity axiom in \cite{FJR}.} 
The {\em concavity axiom} in~\cite[Theorem 4.1.8]{FJR} is in fact the only direct method we use in the paper to calculate FJRW invariants.

Recall that for a nondegenerate quasihomogeneous polynomia $W: \C^N\to \C,$ let $(\mathcal{L}_1, \ldots, \mathcal{L}_N)$ be the universal orbifold line bundles on the universe curve $\mathcal{C}$ where 
$\pi: \mathcal{C}\to \overline{\mathcal{W}}_{g, k}(g_1, \ldots, g_k)$ is the projection to the moduli of $W$-structures. We say a correlator 
$\<\gamma_1, \ldots, \gamma_k\>_{0, k}$ is {\em concave} if all the insertions $\{\gamma_i\in \cH_{g_i}\}$ are narrow and  $H^0(C, \mathcal{L}_j)=0$ for any $1\leq j\leq N$ and each geometric point $[C]\in \overline{\mathcal{W}}_{g, k}(g_1, \ldots, g_k).$
In this case, $\pi_*(\bigoplus_{i=1}^N \mathcal{L}_i)=0$ and $R^1\pi_*(\bigoplus_{i=1}^N \mathcal{L}_i)$ is locally free. 
Let $c_{\rm top}$ be the {\em top Chern class}, ${\rm PD}$ be the Poincar\'e dual, and 
$${\rm st}: \overline{\mathcal{W}}_{g, k}(g_1, \ldots, g_k)\to \overline{\mathcal{M}}_{g,k}$$ be the forgetful morphism. 
Then \cite[Theorem 4.1.8]{FJR} implies
\begin{equation}\label{concavity-axiom}
\<\gamma_1, \ldots, \gamma_k\>_{0, k}
=\int_{\overline{\mathcal{M}}_{g,k}}
\frac{1}{\deg \rm st} {\rm PD} \circ {\rm st}_\ast \circ {\rm PD}^{-1}
\left( c_{\rm top}\left(R^1 \pi_\ast (\bigoplus_i \mathcal{L}_i)\right)^\vee\right).
\end{equation}

\subsubsection{Some criteria for nonvanishing FJRW invariants}

%Now we discuss some vanishing conditions of the FJRW invariants. 
%By \eqref{parity}, the parity of the element is always even, thus the quantum product is commutative.
%By \eqref{complex-degree}, we have $$\deg_\C\alpha_i={2(i-1)\over d}, \quad \deg_\C\beta_j=1-{2\over d}.$$
%As a direct consequence of 
According to \cite[Proposition 2.2.8, Theorem 4.1.8 (1)]{FJR}, we have
\begin{lemma}
[Nonvanishing criteria]
\label{selection+degree}
If the FJRW invariant $\<\gamma_1, \cdots, \gamma_k\>_{0,k}$ does not vanish, then the following conditions holds:
\begin{enumerate}
\item (Selection rule) the underlying moduli space is nonempty and %we have %a degree constraint formula
\begin{equation}
\label{selection-rule}
\deg \cL_j=(k-2){w_j\over d}-\sum_{i=1}^{k}\theta_j(\gamma_i)\in \Z \quad \forall 1\leq j\leq N.
\end{equation}
Here $\theta_j(\gamma_i)=\theta_j(g)$ is given by \eqref{exponent-rational} if $\gamma_i\in \cH_{g}$.
\item (Degree formula) the degrees $\deg_\C\gamma_i$ satisfy
\begin{equation}
\label{virtual-cycle-degree}
\sum_{i=1}^{k}\deg_\C\gamma_i=\widehat{c}_W-3+k. %=k-1-{4\over d}.
\end{equation}
\end{enumerate}
\end{lemma}

%\subsubsection{Invariance properties}
\begin{remark}
[$G_W$-invariance]
According to \cite[Theorem 4.1.8 (10)]{FJR}, the FJRW invariants % $\<\phi_1, \cdots, \phi_k\>_{0,k}$ 
are invariant under the $G_W$-action.
This is called the {\em $G_W$-invariance property}.
\end{remark}
\begin{remark}
[Permutation invariance]
The FJRW invariant  $\<\gamma_1, \cdots, \gamma_k\>_{0,k}$ may not be invariant under permutations action in general, because it is only super commutative, and the parity of broad element could be odd.
However, when $W$ is an invertible polynomial with two variables, all the broad elements in \eqref{broad-generator} have even parity, and the FJRW invariant $\<\gamma_1, \cdots, \gamma_k\>_{0,k}$ for the LG pair $(W, \<J\>)$ are invariant under permutations.
\end{remark}

%\newpage

\subsection{FJRW theory for simple singularities and their variants}
\label{special-cases}

Now we collect some results in \cite{FJR} for {\em simple singularities} (or $ADE$ singularities). We will  use these results to show Conjecture \ref{main-conj} holds for $W$ if $\widehat{c}_W<1.$ %$W=x_1^3+x_2^3$ or $W=x_1^2+x_2^m$ for $m\geq 2$.

\subsubsection{FJRW theory of simple singularities}
\label{sec-simple}
In singularity theory, under right equivalence, a polynomial singularity $W$ with central charge $\widehat{c}_W<1$ must be equivalent to one of the forms below:
\begin{equation}
\label{mirror-ade-singularity}
\begin{dcases}
A_{n}: & W=x^{n+1}, n\geq 1;\\
D_{n}: & W=x^{n-1}+xy^2, n\geq 3;\\
E_6: & W= x^3+y^4;\\
E_7: & W=x^3+xy^3;\\
E_8: & W=x^3+y^5.
\end{dcases}
\end{equation}
These are called simple singularities, or $ADE$ singularities. 
In \cite{FJR, FFJMR}, besides the $D_3$ case, the FJRW theory for the polynomials in \eqref{mirror-ade-singularity} as well as for the $D_{n}^T$-type polynomial $W=x^{n-1}y+y^2$ with $n\geq 3$ have been studied.
All possible admissible groups for these polynomials have been considered. 
We remark that among such polynomials, non-maximal admissible group only exists for $x^{n-1}+xy^2$ when $n$ is even.

We say two FJRW theories are {\em equivalent}, if there is an isomorphism between two underlying state spaces, such that the potentials of the underlying Dubrovin-Frobenius manifolds and the ancestor FJRW potentials are the same. 
We summarize the results in \cite[Theorem 6.1.3]{FJR} and \cite[Theorem 4.5, Lemma 5.2]{FFJMR} as below.
\begin{theorem}
\label{theorem-ade}
\cite{FJR, FFJMR}
If $W$ is a polynomial in \eqref{mirror-ade-singularity}, or $W=x^{n-1}y+y^2, n\geq 3$, then the FJRW theory of $(W, G)$ is semisimple because:
\begin{enumerate}
\item
the Dubrovin-Frobenius manifold of $(W, G_W)$ is equivalent to the Dubrovin-Frobenius manifold in the Saito theory of the mirror polynomial $W^T$;
\item 
if $n$ is even and $n\geq 4$, the Dubrovin-Frobenius manifold of the LG pair $(W=x^{n-1}+xy^2, \<J\>)$ is equivalent to the Dubrovin-Frobenius manifold in the Saito theory of $W=x^{n-1}+xy^2$. 
\end{enumerate}
\end{theorem}

\subsubsection{Conjecture \ref{main-conj} %for Fermat polynomials 
when $\widehat{c}_W<1$}

In general, the FJRW theory of an LG pair $(W, G)$  depends on the choices of both $W$ and $G$, so the equivalence classes are different from the right equivalence classes of singularities.
For the invertible simple singularities, we have 
\begin{theorem}
\label{theorem-simple}
Let $\widetilde{W}$ be an invertible polynomial with two variables such that $\widehat{c}_{\widetilde{W}}<1$, then the FJRW theory of any admissible pair $(\widetilde{W}, G)$ is either the FJRW theory of $(x^2+y^2, \<J\>)$ or equivalent to some $(W, G_W)$ where $W$ is one of the polynomials in \eqref{mirror-ade-singularity}. 
\end{theorem}
 Now part (1) of Theorem \ref{main} follows from Theorem \ref{theorem-ade} and Theorem \ref{theorem-simple}.

\noindent {\em Proof of Theorem \ref{theorem-simple}.}
%We first consider invertible polynomials of chain type or loop type. 
According to the classification, up to permutation, an invertible polynomial $W$ with $N=2$ and $\widehat{c}_W<1$ must be one of the following cases. 
In each case, there are at most two choices of admissible groups, $\<J\>$ and $G_W$.
\begin{enumerate}
\item $W=x^{n-1}+xy^2,$ or $x^{n-1}y+y^2$ with $n\geq3$, or $x^3+xy^3$.
\item $W=x^2+y^2, x^3+y^3, x^3+y^4, x^3+y^5$. 
\item $W=x^m+y^2$ for $m\geq 3$.
\item $W=x^{m}y+xy^2$ with $m\geq 2$.
\end{enumerate}

The first situation is already included in Theorem \ref{theorem-ade}. 

%Now we consider Fermat polynomials with two variables. 
For the second situation, we only need to show the cases when $\<J\>\neq G_W.$ This happens when $W=x^2+y^2$, or $x^3+y^3$.
The FJRW theory of $(W=x^2+y^2, \<J\>)$ is semisimple.
In this case, only classical product exists,
$$\alpha_1\star\alpha_1=\beta_1\star\beta_1=\alpha_1, \quad \alpha_1\star\beta_1=\beta_1\star\alpha_1=\beta_1.$$
The Frobenius manifold is semisimple because there is an idempotent basis $$\left\{e_1={\alpha_1+\beta_1\over 2},  \quad e_2={\alpha_1-\beta_1\over 2}\right\}.$$ 
On the other hand, the FJRW theory of $(x^3+y^3, \<J\>)$ is also semisimple. It is equivalent to the FJRW theory of $(x^3+xy^2, \<J\>)$,  following the method used in \cite{FFJMR}.

%\newpage 
For the third situation, we have
%Now we prove Conjecture \ref{main-conj} for the first case by showing that each of the FJRW theory is equivalent to one of the FJRW theories that appeared in Theorem \ref{theorem-ade}.
\begin{proposition}\label{fermat-small}
We have the following equivalences between FJRW theories.
\begin{enumerate}
\item The FJRW theory of $(x^m+y^2, G_W)$ is equivalent to the FJRW theory of $(x^m, \<J\>)$. In particular, if $m$ is odd, $\<J\>=G_W$.
\item If $m=2n$ is even and $n\geq2$, the FJRW theory of $(x^{2n}+y^2, \<J\>)$ is equivalent to the FJRW theory of $(W=x^{n}y+y^2, \<J\>=G_W)$.
\end{enumerate}
\end{proposition}
\begin{proof}
The first case is easy. Let us check the second case. 
The weight system of each LG pair is $(d; w_1, w_2)=(2n; 1, n).$
Using \eqref{narrow-generator} and \eqref{broad-generator}, each FJRW vector space has a basis of the form
$\{\alpha_1, \alpha_3, \ldots, \alpha_{2n-1}, x^{n-1}dx\wedge dy\}.$
The isomorphism given by $\alpha_i\mapsto \alpha_i$ and $ x^{n-1}dx\wedge dy\mapsto x^{n-1}dx\wedge dy$ preserves the pairing and the potential. 
The reconstruction follows from \cite[Theorem 6.2.10, part (3)]{FJR} and the calculation of the four point correlator $\<\alpha_3, \alpha_3, \alpha_{2n-3}, \alpha_{2n-1}\>_{0,4}$ follows from \cite[Section 6.3.7]{FJR}.
% Find an isomorphism between vector spaces. Then using the reconstruction via WDVV equation, and calculation using concavity.
\end{proof}

For the last situation, we have
\begin{proposition}
\label{loop-proof}
The FJRW theory of $(W=x^{m}y+xy^{2}, \<J\>)$ with $m\geq 2$ is equivalent to the FJRW theory of the LG pair $(W=x^{2m-1}+xy^2, \<J\>)$.
\end{proposition}
\begin{proof}
The argument is similar as Proposition \ref{fermat-small}.
The weight system of each LG pair is $(d; w_1, w_2)=(2m-1; 1, m-1).$
Using \eqref{narrow-generator} and \eqref{broad-generator}, each FJRW vector space has a basis of the form 
$\{\alpha_1, \alpha_2, \ldots, \alpha_{2m-2}, \beta_0, \beta_1\}.$
There is also a natural isomorphism 
\begin{equation}\label{loop-chain-iso}
\cH_{x^{2m-1}+xy^2, \<J\>}\cong \cH_{x^{m}y+xy^{2}, \<J\>}
\end{equation} that preserves the pairing and the potential. 
The isomorphism sends $\alpha_i$ to $\alpha_i$, but the transformation for $\beta_0$ and $\beta_1$ is not the obvious one. 
In fact, by \cite[Theorem 2.5]{Acosta} and \cite[Section 5.2.4]{FJR}, there are Frobenius algebra isomorphisms
$$\cH_{x^{m}y+xy^{2}, \<J\>}\cong {\rm Jac}(x^{m}y+xy^{2}), \quad \cH_{x^{2m-1}+xy^2, \<J\>}\cong {\rm Jac}(x^{2m-1}+xy^2).$$ 
So the natural isomorphism in \eqref{loop-chain-iso} is obtained from these two isomorphisms and the isomorphism 
${\rm Jac}(x^{2m-1}+xy^2)\cong {\rm Jac}(x^{m}y+xy^{2})$ given by $x\mapsto x$ and $y\mapsto x^{m-1}+2y$. 
The identification of the potential under the isomorphism \eqref{loop-chain-iso} follows from the same reconstruction argument used in \cite[Theorem 6.2.10, part (3)]{FJR}.
If $m=2$, it is studied in \cite[Appendix A, Case 1]{HLSW}.
If $m\geq 3$, the single four-point correlator that determines the FJRW theory of the pair $(W=x^{m}y+xy^{2}, \<J\>)$ is $\<\alpha_3, \alpha_3, \alpha_{2m-4}, \alpha_{2m-2}\>_{0,4}$. 
This is a concave correlator. %and $$\deg_{\C}\alpha_3={1\over 2m-1}, \quad \deg_{\C}\alpha_{2m-4}={2m-3\over 2m-1} \quad \deg_{\C}\alpha_{2m-2}={2m-2\over 2m-1}.$$
Using the concavity axiom \eqref{concavity-axiom} and formula (95) in \cite[Theorem 6.3.3]{FJR}, 
its value equals to $\<X, X, X^{2m-2}, X^{2m-3}\>_0$ in  \cite[Theorem 6.2.10, part (3)]{FJR}.
See also the calculation in \cite[Section 6.3.5]{FJR}.
%Then using the reconstruction via WDVV equation, and calculation using concavity.
\end{proof}

%If $d=3$, the FJRW theory of the pair is equivalent to the FJRW theory of $D_4$-singularity, which is semisimple according to \cite{FFJMR}.

\iffalse
We give some examples when $W=x_1^d+x_2^d$, with $d=2, 3, 4$.
\begin{itemize}
\item

If $d=2$, only classical product exists,
$$\alpha_1\star\alpha_1=\beta_1\star\beta_1=\alpha_1, \alpha_1\star\beta_1=\beta_1.$$
The Frobenius manifold is semisimple as we have an \textcolor{red}{idempotent} basis $$\left\{e_1={\alpha_1+\beta_1\over 2}, e_2={\alpha_1-\beta_1\over 2}\right\}.$$ 
%Thus it suffices to check for cases when $d\geq 4$. 
%\item If $d=4$, the quantum product for this LG pair has been studied in \cite[Section 4.2]{Francis}.
\end{itemize}

Collect some earlier results. 
(1) Explain the statement is easy (or maybe hard) in cases there are weight ${1\over 2}$ variables.
(2) When $\<J\>=G_W?$
(3) Other cases when $\<J\>\neq G_W$: $D_n$ when $n$ is odd/even? and examples in \cite{HS}.

\fi

%\newpage

\section{Quantum Euler classes and semisimplicity}
\label{sec-nonmax}

%\subsection{Quantum Euler classes}

%Let $W: \C^2\to\C$ be a nondegenerate invertible polynomial, and $\<J\>$ be the minimal admissible group of $W$. 

From now on, we always assume $W:\C^2\to \C$ is a degree $d$ invertible polynomial with $\widehat{c}_W\geq 1$. 
By assumption, $d\geq 4$ and $W$ contains no weight-${1\over 2}$ variables, so $\alpha_2$ is a narrow element.
We will consider the LG pair $(W, \<J\>)$ with the minimal admissible group $\<J\>$. 

In general, the {\em quantum Euler class} (with respect to the basis $\{\phi_i\}$) is given by 
\begin{equation}
\label{euler-def}
{\bf E}({\bf t}):=\sum_{i=1}^{\mu} \phi_i\star_{\bf t}\phi^i.
\end{equation}
It is the {\em characteristic element} \cite[Section 3]{Abr} of the Frobenius algebra $(\cH_{W, \<J\>}, \star_{\bf t}, \< \cdot, \cdot\>)$. 
In~\cite[Theorem 3.4]{Abr}, Abrams proved that the characteristic element of a Frobenius algebra $A$ is a unit if and only if A is semisimple.
We will use this result to prove Theorem~\ref{main}.

%\subsubsection{Computation of $\LL\phi_i,\phi^i, \phi\RR_{0,3}(t\alpha_2)$}
\subsection{A formula of the quantum Euler class}

%\subsubsection{The quantum product along $\alpha_2$}

Let $t$ be the coordinate of $\alpha_2$. We calculate the quantum product $\star_{t\alpha_2}$ and the restriction of the quantum Euler class in~\eqref{euler-def} to ${\bf t}=t\alpha_2$.
Recall $$\LL\phi_i,\phi_j, \phi_k\RR_{0,3}(t\alpha_2)=\sum_{n\geq 0}{t^n\over n!}\<\phi_i, \phi_j, \phi_k, \alpha_2, \ldots, \alpha_2\>_{0, n+3}.$$
In order to compute the quantum Euler class, it is enough to consider the cases when $\phi_j=\phi^i$. 
\begin{lemma}
%Let $W: \C^2\to \C$ be an invertible polynomial with no weight-${1\over 2}$ variables. 
The restriction of the quantum Euler class ${\bf E}({\bf t})\vert_{{\bf t}=t\alpha_2}$ is
\begin{equation}
\label{quantum-euler-j2}
{\bf E}(t\alpha_2)={t^2\over 2!}\sum_i B_i\alpha_1+t\sum_{k=0, \delta}\sum_i A_{i,k}\beta^k + \mu\alpha_{d-1},
\end{equation}
where
\begin{equation}
\label{quantum-euler-vector-explicit}
A_{i,k}:=\<\phi_i, \phi^i, \beta_k, \alpha_2\>_{0,4}; \quad
B_i:=\<\phi_i,\phi^i, \alpha_{d-1}, \alpha_2, \alpha_2\>_{0,5}.
\end{equation}
\end{lemma}
\begin{proof}
It is easy to see $$\deg\phi_i+\deg\phi^i=\widehat{c}_W, \quad \deg\alpha_2={w_1+w_2\over d}.$$
Applying the degree formula \eqref{virtual-cycle-degree}, the nonzero coefficients in $\LL\phi_i,\phi^i, \phi\RR_{0,3}(t\alpha_2)$ can only happen when 
$$\deg_\C\phi={n\over 2}\cdot \widehat{c}_W, \quad \text{for} \quad n=0,1,2.$$
There are three cases:
\begin{enumerate}
\item if $n=0$, then $\phi_k=\alpha_1$ and 
$\LL\phi_i,\phi^i, \alpha_1\RR_{0,3}(t\alpha_2)=1.$
\item if $n=1$, either $\phi$ is broad or $\phi$ is narrow.
\item if $n=2$, then $\phi=\alpha_{d-1}$ and 
$$\LL\phi_i,\phi^i, \alpha_{d-1}\RR_{0,3}(t\alpha_2)={t^2\over 2!}\<\phi_i,\phi^i, \alpha_{d-1}, \alpha_2, \alpha_2\>_{0,5}.$$
\end{enumerate}
It remains to consider the cases when $n=1$. In fact, using the selection rule \eqref{selection-rule}, we see $\phi$ must be broad. 
However, the {\em $G_W$-invariance} property implies that if $\beta_k\notin \cH_{W, \<J\>}^{G_W}$, then 
$$\<\alpha_i, \alpha^i, \beta_k, \alpha_2\>_{0,4}=\<\beta_i, \beta^i, \beta_k, \alpha_2\>_{0,4}=0.$$
If  $\beta_k\in \cH_{W, \<J\>}^{G_W}$, then by Lemma \ref{lemma-invariant-subspace}, it must be $\beta_0$ if $W$ is a $2$-chain, $\beta_0$ or $\beta_\delta$ if $W$ is a $2$-loop. 
Now the result follows from~\eqref{euler-def}.
\end{proof}

%\begin{remark} Can we calculate these invariants using the LG model of $(W, G_W)$? \end{remark}
%We can have some nontrivial genus-0 4-point FJRW invariant. 
\iffalse
We give some examples of possibly non-vanishing genus-0 4-point invariants when $W$ is a chain polynomial in \cite{Francis}.
{\red
\begin{itemize}
\item $W=x^3+xy^8$, $\<Z, Z, Z, Y=\alpha_2\>$, $\<W,W,Z,Y\>$. But they cancel with each other if we combine them appropriately.
On the other hand, those with one broad insertion, $\<X, X^2Y , Z, Y\>$,  $\< X^2, XY, Z, Y\>$,  $\< X^3, Y, Z, Y\>$, seem to be zero.
\item $W=x^3+xy^6$, $\<Z, Z, Z, Y\>$, $\<W, W, Z,Y\>$. They do not seem to cancel each other. What happens?
\item $W=x^3y+y^7$, $\<Z, W, X^2, Y\>, \<X, Y^2, X^2, Y\>, \<Y, XY, X^2, Y\>, \<X^2, X^2, X^2, Y\>$. They don't cancel.
\end{itemize}
}

\begin{remark}
 If $W$ is of Brieskorn-Pham type,  there is no $\beta^0$ term and the quantum Euler class is easy to calculate.
 If $W$ is of chain type or loop type, the best scenario is that there are some cancellation happening between the coefficients of $\beta^0$. 
\begin{itemize}
\item In fact, this is the case when $W=x^3+xy^8$. 
\item For $W=x^3+xy^6$, there are some typo in \cite{Francis}? and the cancellation should hold. 
\item For $W=x^3y+y^7$, it seems there is no cancellation.
\end{itemize}
\end{remark}
\fi

\subsubsection{A genus-$0$ $5$-point correlator}

We investigate the nonvanishing contributions on the quantum Euler class~\eqref{quantum-euler-j2}.
We have an explicit formula for the following genus-$0$ $5$-point correlator.
\begin{proposition}
\label{nonvanishing-appendix}
%[A. E. Francis]
Let $W:\C^2\to \C$ be a degree $d$ invertible %Fermat
polynomial with $\widehat{c}_W\geq 1$. %$d\geq 4$ and the weights $wt(x_i)={w_i\over d}\neq {1\over 2}, i=1, 2.$
For the LG pair $(W, \<J\>)$,
%the FJRW invariant $C_d:=\<\alpha_{d-1},  \alpha_{d-2},  \alpha_{2},  \alpha_{2},  \alpha_{2}\>_{0,5}$ is nonzero.
we have a nonzero FJRW invariant
\begin{equation}
\label{definition-5-point}
C_d:=\<\alpha_{d-1},  \alpha_{d-2},  \alpha_{2},  \alpha_{2},  \alpha_{2}\>_{0,5}={w_1w_2\over d^2}.
\end{equation}
 \end{proposition}
 This correlator is concave and thus can be calculated by the formula~\eqref{concavity-axiom}. The calculation is quite complicated and will be presented in Section \ref{5-point-appendix}.
% We expect this equality \eqref{definition-5-point} also holds if $W$ is of chain type or loop type.

% Sometimes it is easy to find the value of this correlator using WDVV equations. Let us give two examples, one in Fermat case, another in chain case.

\subsubsection{Quantum Euler classes for Brieskorn-Pham polynomials}
When $W$ is a Brieskorn-Pham polynomial, we can simplify the formula~\eqref{quantum-euler-j2} further and get an explicit formula.
According to Lemma \ref{lemma-invariant-subspace}, if $W$ is a Brieskorn-Pham polynomial, the $G_W$-invariant subspace $\cH_{W, \<J\>}^{G_W}\subset \cH_{W, \<J\>}$ contains no broad element. So the second term in~\eqref{quantum-euler-j2} vanishes.
\begin{proposition}
Let $W=x_1^{a_1}+x_2^{a_2}=x_1^{\delta w_2}+x_2^{\delta w_1}$ be a Brieskorn-Pham polynomial with $\gcd(w_1, w_2)=1$ and $\widehat{c}_W\geq1$, then %$\delta w_1>2$ and $\delta w_2>2$, then %the quantum Euler class~\eqref{quantum-euler-j2} takes the form
$${\bf E}(t\alpha_2)=(\delta w_1 w_2-w_1-w_2-\delta)C_d {t^2\over 2!} \alpha_1+\mu\alpha_{d-1}.$$
\end{proposition}
It suffices to compute the first term in~\eqref{quantum-euler-j2}.
The result follows from two lemmas below.
%The coefficients $C_i$ and $D_j$ can be related by WDVV equations.
%We denote the following genus-0 5-point FJRW invariant by $$C_d:=\<\alpha_{d-1},  \alpha_{d-2},  \alpha_{2},  \alpha_{2},  \alpha_{2}\>_{0,5}.$$

\begin{lemma}
[Narrow comtribution]
\label{lemma-narrow}
For any $i\in {\bf Nar}\backslash\{1, 2, d-2, d-1\}$, we have
$$\<\alpha_{d-1},  \alpha_{d-i},  \alpha_{i},  \alpha_{2},  \alpha_{2}\>_{0,5}=C_d.$$
\end{lemma}
\begin{proof}
%When $W$ is of Fermat type, if $k\notin{\bf Nar}$, using the definition of ${\bf Nar}$ in~\eqref{narrow-index},  we have $\delta w_1\mid k$ or $\delta w_2\mid k$ (recall that $w_1$ and $w_2$ are coprime). 
%So the difference of two broad indices is at least $2.$
For $i\in {\bf Nar}\backslash\{1, 2, d-2, d-1\}$, at least one of the following cases is true.
 \begin{enumerate}
 \item $i-1\in {\bf Nar}$.
 \item $i+1\in {\bf Nar}$.
 \item $i-1, i+1\notin{\bf Nar}$.
 \end{enumerate}
 
 For the first case, we replace $(\gamma_1, \gamma_2, \gamma_3, \gamma_4)$ by $(\alpha_{d-1}, \alpha_{d-i}, \alpha_2, \alpha_{i-1})$ in~\eqref{wdvv-equation}, and obtain
\begin{align*}
\sum_{a}\LL\alpha_{d-1}, \alpha_{d-i}, \phi_a\RR_{0,3}(t) \LL \phi^a, \alpha_2, \alpha_{i-1}\RR_{0,3}(t)\\
=
\sum_{a}\LL\alpha_{d-1}, \alpha_2, \phi_a\RR_{0,3}(t)\LL \phi_a, \alpha_{d-i}, \alpha_{i-1}\RR_{0,3}(t).
\end{align*}
We compare the coefficients of $t^2$ on both sides. 
Let us consider all possible nontrivial contributions. 
By the $G_W$-invariance property, $\phi_a$ can not be a broad element as broad element are not $G_W$ invariant for the Brieskorn-Pham polynomial. 
By the selection rule~\eqref{selection-rule}, $\theta_j(\phi_a)$ is uniquely fixed by the rest of insertions. 
So there is at most one choice of narrow element $\phi_a$ which could make nontrivial contribution.
By the degree formula~\eqref{virtual-cycle-degree},  we have $\<\alpha_{d-1}, \alpha_{d-i}, \phi_a\>_{0,3}=0$ since $i\neq d-1$.
So the contribution of $t^2$-term on the left-hand side is given by
\begin{align*}
&{1\over 2}\<\alpha_{d-1},  \alpha_{d-i},  \alpha_{i},  t\alpha_{2},  t\alpha_{2}\>_{0,5}\< \alpha_{d-i},  \alpha_{2},  \alpha_{i-1}\>_{0,3}\\
&+\<\alpha_{d-1},  \alpha_{d-i},  \alpha_{i+1}, t\alpha_{2}\>_{0,4}\<\alpha_{d-i-1},  \alpha_{2},  \alpha_{i-1}, t\alpha_{2}\>_{0,4}+0.
%\\=&{t^2\over 2}\<\alpha_{d-1},  \alpha_{d-i},  \alpha_{i},  \alpha_{2},  \alpha_{2}\>_{0,5}\times 1+0\times 0+0.
\end{align*}
Since $i-1\in {\bf Nar}$, $\<\alpha_{d-i},  \alpha_{2},  \alpha_{i-1}\>_{0,3}=1$ follows from concavity axiom, in which case, the underlying moduli space is just a point. 
For the two genus-$0$ $4$-point correlators, if $i+1\notin {\bf Nar}$, then they vanish by $G_W$-invariance property. 
If $i+1\in {\bf Nar}$, then they vanish by checking the degree formula~\eqref{virtual-cycle-degree}.
So the contribution of $t^2$-term on the left-hand side is given by ${t^2\over 2}\<\alpha_{d-1},  \alpha_{d-i},  \alpha_{i},  \alpha_{2},  \alpha_{2}\>_{0,5}.$

Similarly, the contribution of the $t^2$-term on the right-hand side is given by
\begin{align*}
&{1\over 2}\<\alpha_{d-1},  \alpha_{2},  \alpha_{d-2},  t\alpha_{2},  t\alpha_{2}\>_{0,5}\< \alpha_{2},  \alpha_{d-i},  \alpha_{i-1}\>_{0,3}\\
&+\<\alpha_{d-1},  \alpha_{2},  \alpha_{d-1}, t\alpha_{2}\>_{0,4}\<\alpha_{1},  \alpha_{d-i},  \alpha_{i-1}, t\alpha_{2}\>_{0,4}+0\\
=&{t^2\over 2}\<\alpha_{d-1},  \alpha_{2},  \alpha_{d-2},  \alpha_{2},  \alpha_{2}\>_{0,5}.
\end{align*}
Thus we have
$$\<\alpha_{d-1},  \alpha_{d-i},  \alpha_{i},  \alpha_{2},  \alpha_{2}\>_{0,5}=\<\alpha_{d-1},  \alpha_{2},  \alpha_{d-2},  \alpha_{2},  \alpha_{2}\>_{0,5}=C_d.$$

For the second case, we recognize that both $d-i, (d-i)-1\in {\bf Nar}$. So we can use the argument in the first case for $\alpha_{d-i}$ instead of for $\alpha_i$, and obtain the same result.

For the last case, we replace $(\gamma_1, \gamma_2, \gamma_3, \gamma_4)$ by $(\alpha_{i}, \alpha_{d-i}, \alpha_{2}, \alpha_{d-2})$ in the WDVV equation~\eqref{wdvv-equation}, and obtain
\begin{align*}
\sum_{a}\LL\alpha_{i}, \alpha_{d-i}, \phi_a\RR_{0,3}(t) \LL \phi^a, \alpha_2, \alpha_{d-2}\RR_{0,3}(t)\\
=
\sum_{a}\LL\alpha_{i}, \alpha_2, \phi_a\RR_{0,3}(t)\LL \phi_a, \alpha_{d-2}, \alpha_{d-i}\RR_{0,3}(t).
\end{align*}
Again, we compare the coefficients of $t^2$ on both sides. 
After deleting the vanishing correlators, we obtain a formula
\begin{align*}
&{1\over2}\<\alpha_{i}, \alpha_{d-i},  \alpha_{d-1},  \alpha_{2},  \alpha_{2}\>_{0,5}+{1\over2}\<\alpha_{d-1},  \alpha_{2},  \alpha_{2},  \alpha_{d-2},  \alpha_{2}\>_{0,5}\\
=&\<\alpha_{i},  \alpha_{2}, \alpha_{d-2-i},  \alpha_{2}\>_{0,4}\<\alpha_{i+2},  \alpha_{d-2},  \alpha_{d-i},  \alpha_{2}\>_{0,4}.
\end{align*}
Since $i-1, i+1\notin{\bf Nar}$, each element in $\{i-1 ,i+1\}$ is either divisible by $\delta w_1$ or $\delta w_2$. 
Recall that we assume $\delta w_1>2$ and $\delta w_2>2$.
So neither $i-1, i+1$ can be divisible by both $\delta w_1$ or $\delta w_2$. Without loss of generality, we assume $a_2=\delta w_1\mid (i-1)$ and $a_1=\delta w_2\mid (i+1)$.
So the two genus-$0$ $4$-point correlators on the right-hand side are both concave. 
More explicitly, we have 
$$\begin{dcases}
\Big(\theta_1(\alpha_{i}),  \theta_1(\alpha_{2}), \theta_1(\alpha_{d-2-i}),  \theta_1(\alpha_{2})\Big)= \Big({a_1-1\over a_1}, {2\over a_1}, {a_1-1\over a_1}, {2\over a_1}\Big), \\
\Big(\theta_2(\alpha_{i}),  \theta_2(\alpha_{2}), \theta_2(\alpha_{d-2-i}),  \theta_2(\alpha_{2})\Big)= \Big({1\over a_2}, {2\over a_2}, {a_2-3\over a_2}, {2\over a_2}\Big).
\end{dcases}
$$
We have $(\deg \cL_1, \deg \cL_2)=(-2, -1).$ 
Following the calculation in~\cite[Lemma 6.6]{HLSW}, the concavity axiom implies
$$\<\alpha_{i},  \alpha_{2}, \alpha_{d-2-i},  \alpha_{2}\>_{0,4}={1\over a_1}={w_1\over d}.$$
Similarly, we have 
$$\<\alpha_{i+2},  \alpha_{d-2},  \alpha_{d-i},  \alpha_{2}\>_{0,4}={1\over a_2}={w_2\over d}.$$
For the other case, $a_2=\delta w_1\mid (i+1)$ and $a_1=\delta w_2\mid (i-1)$, the values of two genus-$0$ $4$-point correlators are ${w_2\over d}$ and ${w_1\over d}.$
So the product is always ${w_1w_2\over d^2}$ and the result follows.
\end{proof}

%\newpage
\begin{lemma}
[Broad contribution]
\label{lemma-broad}
For any $\beta_j=x_1^{j w_2-1}x_2^{(\delta-j) w_1-1}dx_1\wedge dx_2$, we have
$$\<\beta_j, \beta^j, \alpha_{d-1}, \alpha_2, \alpha_2\>_{0,5}=-C_d.$$
\end{lemma}
\begin{proof}
We replace $(\gamma_1, \gamma_2, \gamma_3, \gamma_4)$ by $(\beta_j, \beta^j, \alpha_2, \alpha_{d-2})$ in the WDVV equation~\eqref{wdvv-equation}, and obtain
\begin{align*}
\sum_{a}\LL\beta_j, \beta^j, \phi_a\RR_{0,3}(t)\cdot \LL \phi^a, \alpha_2, \alpha_{d-2}\RR_{0,3}(t)
\\
=
\sum_{a}\LL\beta_j, \alpha_2, \phi_a\RR_{0,3}(t)\cdot\LL \phi_a, \beta^j, \alpha_{d-2}\RR_{0,3}(t).
\end{align*}
We can compare the coefficients of $t^2$ on both sides and obtain
%For the left-hand side, the argument in Lemma \ref{three-broad-calculation} shows it is
%So we obtain {\red (be careful about the nonvanishing $4$-point and $6$-point correlators)}
\begin{align*}
%=&\sum_{a}\LL\beta_j, \beta^j, \phi_a\RR_{0,3}(t)\cdot \LL \phi^a, \alpha_2, \alpha_{d-2}\RR_{0,3}(t)\\
%&\LL\beta_j, \beta^j, \alpha_{d-1}\RR_{0,3}(t)\cdot\LL \alpha_1, \alpha_2, \alpha_{d-2}\RR_{0,3}(t)+\LL\beta_j, \beta^j, \alpha_{1}\RR_{0,3}(t)\cdot\LL \alpha_{d-1}, \alpha_2, \alpha_{d-2}\RR_{0,3}(t)\\
{t^2\over 2!}\left(\<\beta_j, \beta^j, \alpha_{d-1}, \alpha_2, \alpha_2\>_{0,5}+\<\alpha_{d-1},   \alpha_{2}, \alpha_{d-2},  \alpha_{2},  \alpha_{2}\>_{0,5}\right)=0.
\end{align*}
%$$\LL\beta_j, \beta^j, \alpha_{d-1}\RR_{0,3}(t)\cdot\LL \alpha_1, \alpha_2, \alpha_{d-2}\RR_{0,3}(t)+\LL\beta_j, \beta^j, \alpha_{1}\RR_{0,3}(t)\cdot\LL \alpha_{d-1}, \alpha_2, \alpha_{d-2}\RR_{0,3}(t).$$
%When $W=x_1^{a_1}+x_2^{a_2}$,  \textcolor{red}{the right-hand side vanishes according to ???.} 
%Thus we obtain 
%\begin{equation}\label{broad-nonvanishing}\<\beta_j, \beta^j, \alpha_{d-1}, \alpha_2, \alpha_2\>_{0,5}=-\<\alpha_{d-1},  \alpha_{d-2},  \alpha_{2},  \alpha_{2},  \alpha_{2}\>_{0,5}. \end{equation}
Now the result follows from the definition of $C_d$ in~\eqref{definition-5-point}.
\end{proof}

\subsection{A proof of Theorem \ref{main}}
Now we calculate the quantum multiplication of ${\bf E}(t\alpha_2)\star_{t\alpha_2}$ for the Brieskorn-Pham case. %We prove it is a unit when $t\neq 0$.

\subsubsection{The multiplication of $\alpha_{d-1}$}
%By definition, the quantum multiplication $\alpha_{d-1}\star_{t\alpha_2}$ is determined by $\LL\alpha_{d-1}, \phi_i, \phi_j \RR_{0,3}(t\alpha_2)$. 
\begin{lemma}
%We consider $\LL\alpha_{d-1}, \phi_i, \phi_j \RR_{0,3}(t\alpha_2)$ for the LG pair $(W=x_1^{a_1}+x_2^{a_2}, \<J\>)$. 
The nonvanishing coefficients of $t^n$ in $\LL\alpha_{d-1}, \phi_i, \phi_j \RR_{0,3}(t\alpha_2)$ come from the following cases:
\begin{enumerate}
\item $n=0$, $\phi_i=\phi_j=\alpha_1$.
\item $n=2$, $\phi_j=\phi^i.$
\item $n=4$, $\phi_i=\phi_j=\alpha_{d-1}.$
%\item $n=1$ or $3$.
\end{enumerate} 
\end{lemma}
\begin{proof}
Again, the result follows from the selection rule~\eqref{selection-rule}, the degree formula~\eqref{virtual-cycle-degree}, and the $G_W$-invariance property. 
In fact, if the correlator $\<\alpha_{d-1}, \phi_i, \phi_j, \alpha_2, \ldots, \alpha_2\>_{n+3}$ does not vanish, the degree formula shows
$${n\cdot \widehat{c}_W\over 2}=\deg_\C\phi_i+\deg_\C\phi_j\leq 2 \cdot\widehat{c}_W.$$
So $n\leq 4$. Let us discuss in cases.
(1)
If $n=0$ or $n=4$, the result follows from~\eqref{degree-homogeneous}. 
(2)
If $n=1$, the selection rule implies $\phi_i, \phi_j$ 
must satisfy 
$$\theta_k(\phi_i)+\theta_k(\phi_j)\equiv {a_k-1\over a_k}\mod \mathbb{Z}, \quad \forall k=1, 2.$$
Together with $G_W$-invariance property, this implies that both $\phi_i, \phi_j$ are narrow elements and we obtain equalities for the degree of lines bundles $\deg\cL_1=\deg \cL_2=-2$. 
This will violate the degree formula.
(3)
If $n=2$, the selection rule implies if $\phi_i\in\cH_g$, then $\phi_j$ must be in $\cH_{g^{-1}}$. Furthermore, when both $\phi_i, \phi_j$ are broad elements, the $G_W$-invariance property guarantees that they must be dual to each other.
(4)
If $n=3$, the selection rule implies $\phi_i, \phi_j$ 
must satisfy 
$$\theta_k(\phi_i)+\theta_k(\phi_j)\equiv {a_k+1\over a_k}\mod \mathbb{Z}, \quad \forall k=1, 2.$$
This also implies  $\deg \cL_1=\deg \cL_2=-2$, which violates the degree formula.
\end{proof}

According to this lemma, we have
\begin{align*}
\alpha_{d-1}\star_t\alpha_{2}
&=\sum_{a}\LL\alpha_{d-1}, \alpha_{2}, \phi_a\RR_{0,3}(t)\phi^a\\
&={1\over 2!}\<\alpha_{d-1}, \alpha_{2}, \alpha_{d-2}, t\alpha_2, t\alpha_2\>_{0,5}\alpha_{d-2}.
%={C_d t^2\over 2!}\alpha_2.
\end{align*}

% for any $2\leq i\leq d-2$, we have
%$$\<\alpha_{i},  \alpha_{d-i}, \alpha_{d-1}, \alpha_{2},  \alpha_{2}\>_{0,5}=\<\alpha_{d-1},  \alpha_{d-2},  \alpha_{2},  \alpha_{2},  \alpha_{2}\>_{0,5}=C.$$

Replacing $(\gamma_1, \gamma_2, \gamma_3, \gamma_4)$ by $(\alpha_{d-1}, \alpha_{d-1}, \alpha_2, \alpha_{d-2})$ in~\eqref{wdvv-equation}, we obtain
\begin{align*}
\sum_{a}\LL\alpha_{d-1}, \alpha_{d-1}, \phi_a\RR_{0,3}(t) \LL \phi^a, \alpha_2, \alpha_{d-2}\RR_{0,3}(t)\\
=\sum_{a}\LL\alpha_{d-1}, \alpha_2, \phi_a\RR_{0,3}(t) \LL \phi^a, \alpha_{d-1}, \alpha_{d-2}\RR_{0,3}(t).
\end{align*}
This implies
\begin{equation}\label{7-point}
{1\over 4!}\<\alpha_{d-1}, \alpha_{d-1}, \alpha_{d-1}, \alpha_{2},  \alpha_{2},  \alpha_{2},  \alpha_{2}\>_{0,7}=\left({1\over 2!}\<\alpha_{d-1},  \alpha_{d-2},  \alpha_{2},  \alpha_{2},  \alpha_{2}\>_{0,5}\right)^2.
\end{equation}

%Finally, equation \eqref{quantum-product-five} follows from Lemma \ref{two-broad-calculation} and equation \eqref{broad-nonvanishing}.

By string equation \cite[Theorem 4.2.9]{FJR}, we have
\begin{equation}\alpha_1\star_{t}\phi=\phi, \quad \forall \phi\in \cH_{W, \<J\>};\label{quantum-product-one}
\end{equation}
Using Lemma~\ref{lemma-narrow}, Lemma~\ref{lemma-broad}, and formula~\eqref{7-point}, we have
\begin{proposition}
\label{quantum-product-narrow}
%When $d\geq 4$ and $W$ has no weight-${1\over 2}$ variable, 
Let  $W:\C^2\to \C$ be a degree $d$ invertible Brieskorn-Pham polynomial with $\widehat{c}_W\geq 1$, then
\begin{eqnarray}
%&&\beta_{j}\star_t \beta^j=-{C_dt^2\over 2!}\alpha_1+\alpha_{d-1}; \label{quantum-product-three}\\
%&&\alpha_i\star_t \alpha^i={C_dt^2\over 2!}\alpha_1+\alpha_{d-1}, \quad 2\leq i\leq d-2; \label{quantum-product-two}\\
&&\alpha_{d-1}\star_t \alpha_i=
\begin{dcases}
\alpha_{d-1}, & i=1;\\
{C_d t^2\over 2!}\alpha_i, & i\in {\bf Nar}\backslash\{1, d-1\};\\
{4\choose 2}{C_d^2t^4\over 4!}\alpha_1, & i=d-1.
\end{dcases}
\label{quantum-product-four}
\\
&&\alpha_{d-1}\star\beta_j=-{C_d t^2\over 2!}\beta_j. \label{quantum-product-five}
\end{eqnarray}
\end{proposition}

\subsubsection{A proof of Theorem \ref{main}}
%\subsubsection{Quantum multiplication of  the quantum Euler class}

%We denote by 
Let $\lambda:=\delta w_1 w_2-w_1-w_2-\delta.$
Using \eqref{quantum-product-one}, \eqref{quantum-product-four} and \eqref{quantum-product-five}, the multiplication ${\bf E}(t\alpha_2)\star_{t\alpha_2}$ on the basis
$$\{\alpha_1, \alpha_{d-1}, \alpha_2 \ldots, \alpha_{d-2}, \beta_1, \ldots,\beta_{d-1}\}$$
is given by the matrix
$$\begin{pmatrix}
\lambda C_d {t^2\over 2!}& \mu {4\choose 2}C_d^2{t^4\over 4!}\\
\mu &\lambda C_d {t^2\over 2!}
\end{pmatrix}\oplus\bigoplus_{m\in {\bf Nar}\backslash\{1, d-1\}} (\lambda+\mu) C_d {t^2\over 2!}\cdot {\rm Id}_{\cH_{J^m}}\bigoplus
(\lambda-\mu) C_d {t^2\over 2!}\cdot {\rm Id}_{\cH_{\bf bro}}.$$
By Proposition \ref{nonvanishing-appendix}, $C_d\neq 0$.
%Recall that up to permutation, we assume $W$ is neither $x_1^3+x_2^{15}$, nor $x_1^5+x_2^{20}$. 
%So $(\delta, \{w_1, w_2\})\neq (3, \{1,5\})$ or $(5, \{1, 4\})$.
%Thus by Lemma \ref{delta-mu}, when $t\neq0$, we have %\textcolor{red}{(discuss carefully)}
%\begin{equation}\label{nonvanishing-euler}
Thus we have $\det {\bf E}(t\alpha_2)\star_{t\alpha_2}\neq0$ and
%\end{equation}
 ${\bf E}(t\alpha_2)\star_{t\alpha_2}$ is a unit of the Frobenius algebra $(\cH_{W, \<J\>}, \star_{t\alpha_2}, \<\cdot,\cdot\>)$. % in a small neighborhood of ${\bf t}=(0, t, 0, \cdots, 0)$. 
According to \cite[Theorem 3.4]{Abr}, $(\cH_{W, \<J\>}, \star_{t\alpha_2})$ is semisimple. Thus we complete a proof of Theorem \ref{main}.
\qed
%\newpage
\section{Genus-0 5-point FJRW correlators}
\label{5-point-appendix}

%We give a formula for computing concave genus-0 5-point correlators in FJRW theory. % for the case that the correlator is {\it concave}.  As an application we prove Proposition \ref{nonvanishing-appendix}.

%  As an application we calculate some specific correlators with potential $W = x^d + y^d$ using the symmetry group generated by the differential grading operator $J$. These correlators are necessary to verify that the Virasoro, Polischuk-Vaintrob, and FJRW structures are equivalent in this case. 

\subsection{Concave genus-0 5-point correlators}
In this section, We follow the ideas in {\cite{Francis}}, to give a method for computing concave genus-$0$ $5$-point correlators for any admissible LG pairs $(W, G)$.  
$$D = \sum_i {\rm rank}\, \pi_*\mathcal{L}_i=2;$$ 
thus for possibly non-vanishing correlators, we must have one of the following two cases: 
\begin{enumerate}
    \item[(A)] ${\rm rank}\, \pi_*\mathcal{L}_{i_1} = {\rm rank}\, \pi_*\mathcal{L}_{i_2}=1$, and the ranks of all other pushforward bundles are~0,
    \item[(B)] ${\rm rank}\, \pi_*\mathcal{L}_{i_0} =2$, and the ranks of all other pushforward bundles are 0.
\end{enumerate}

Next, we follow the method outlined in Section 3.1 of \cite{Francis} for giving $c_D(R^1\pi_\ast( \bigoplus_j \mathcal{L}_j) )$ in terms of small degree chern characters. {\it Note:} There is a typo in that paper; on page 1349, Equation 3 should be 
\[
c_t(\mathscr{E}) = \frac{1}{c_t(-\mathscr{E})}=\frac{1}{1-(1-c_t({-\mathscr{E}))}}=\sum_{i=0}^\infty(1-c_t({-\mathscr{E}}))^i,
\]
 yielding
 \[
 c_t(R^1\pi_\ast( \bigoplus_j \mathcal{L}_j) )=\sum_{i=0}^\infty(1-c_t(R^{\bullet}\pi_\ast( \bigoplus_j \mathcal{L}_j) )^i.
 \]
 With this error corrected, it is straightforward to obtain: 
\begin{proposition}
Under the hypotheses of Case A (above):
\begin{equation}\label{eq:cD}
 c_2(R^1\pi_\ast( \bigoplus_j \mathcal{L}_j) )= 
   ch_1(R^\bullet \pi_\ast(\mathcal{L}_{i_1} )) ch_1(R^\bullet \pi_\ast(\mathcal{L}_{i_2} )).
\end{equation}
Under the hypotheses of Case B (above):
\begin{equation}\label{eq:cD}
 c_2(R^1\pi_\ast( \bigoplus_j \mathcal{L}_j) = 
 \frac12  ch_1^2(R^\bullet \pi_\ast(\mathcal{L}_{i_0} ))+ch_2(R^\bullet \pi_\ast(\mathcal{L}_{i_0} )).
\end{equation}
\end{proposition}

Next, we want to describe these low degree chern characters in terms of certain well-understood cohomology classes in $\mathcal{M}_{0,k}$, briefly defined here. \\

\noindent {\it $\psi$-classes}\\
The psi class $\psi_i \in H^*(\overline{\mathcal{M}}_{0,k})$ is the first Chern class of the line bundle whose fiber at $(C,p_1, \ldots, p_k)$ is the cotangent space to $C$ at $p_i$. \\

\noindent {\it $\kappa$-classes}\\
Let $\pi:\overline{\mathcal{M}}_{0,k+1} \to  \overline{\mathcal{M}}_{0,k}$ be the map obtained by forgetting the $k+1$st point on the curve (note that this map also gives the universal curve over $\overline{\mathcal{M}}_{0,k}$). Further, let $\sigma_i:\overline{\mathcal{M}}_{0,k} \to \overline{\mathcal{M}}_{0,k+1} $ be the map sending $(C,p_1, \ldots, p_k)$ to a the nodal curve created by attaching a genus-0, 3-point curve to  $C$ at the point $p_i$ with the (non-nodal) marked points labeled $p_i$ and $p_{k+1}$. Let $D_i\subset \overline{\mathcal{M}}_{0,k+1} $ be the image of $\sigma_i$, and let $K $ be the first chern class of $\omega(\sum_i D_i)$. Then for $a = 1,\ldots, k-3$, we set $\kappa_a = \pi_*(K^{a+1})$.  \\

\noindent {\it Boundary classes}\\
For any partition $K \sqcup K' = [k]:=\{1,\ldots, k\}$ where $2<|K| <k-2$, we have a boundary divisor $\Delta_K$. For simplicity of notation, we will always assume that $1 \in K$. 
Let $\mathcal{C}$ be a genus-0 $k$-point orbicurve with marked points $p_1, p_2, \ldots, p_k$ corresponding to the FJRW elements $\phi_1, \ldots, \phi_k$ ($\phi_i = \alpha_i | \gamma_i\rangle$) Then, we can construct orbicurves $\mathcal{C}_{0,|K|+1} \in \mathcal{M}_{0,|K|+1}$ and $\mathcal{C}_{0,|K|+1} \in \mathcal{M}_{0,|K'|+1}$ such that the marked points of $\mathcal{C}_{0,|K|+1} $ are $p_+ \cup \{p_i\}_{i \in K}$, the marked points of $\mathcal{C}_{0,|K'|+1}$ are $p_- \cup \{p_i\}_{i \in K'}$. We associate the FJRW elements  $\phi_+ = 1|\gamma_+\rangle$, and $\phi_- = 1|\gamma_-\rangle$ with $p_+$ and $p_-$, respectively (concavity implies that $\gamma_+$ and $\gamma_-$ must also be narrow). 
 We can visualize this construction by``cutting" the decorated dual graph of the curve $\mathcal{C}$, in such a way that we get a ``cut'' graph of the form, as seen in 
Figure~\ref{fig:cutgraphs} for the case $k=5$. In that figure, $K$ is either $\{a,b\}$ or $\{c,d,e\}$, whichever contains the element $p_1$. Then let   $\rho_K$ be the canonical glueing map  $\rho_K: \mathcal{M}_{0,|K|+1} \times \mathcal{M}_{0,|K'|+1} \to \mathcal{M}_{0,k}$, obtained by identifying $p_+$ and $p_-$. We can think of this as glueing the dual graph pieces back together. We will denote by  $\psi_+$  the $\psi$ class in $H^\ast (\mathcal{M}_{0,|K|+1} )$ associated to the marked point $p_+$  ($\psi_-$ is defined analogously). 

\begin{figure}[h]
 \begin{center}
    \begin{tikzpicture}
    \node at (-2.2,0) (1) {$c$};
    \node at (-2,1) (2) {$b$};
    \node at (-2,-1) (3) {$d$};
    \node at (2,1) (4) {$1$};
    \node at (2,-1) (5) {$a$};
    \node at (-1,0) (6) {};
    \node at (-.05,0) (7) {};
    \draw (1) -- (-.75,0) -- (-.05,0);
    \draw (2) -- (-.75,0);
    \draw (3) -- (-.75,0);
   \draw (4) -- (.75,0) -- (.05,0);
    \draw (5) -- (.75,0);
    \draw (-.45,.2) node {$+$};
    \draw (.45,.2) node {$-$};
    \end{tikzpicture}
    \qquad \qquad \qquad 
    \begin{tikzpicture}
    \node at (-2.2,0) (1) {$a$};
    \node at (-2,1) (2) {$1$};
    \node at (-2,-1) (3) {$b$};
    \node at (2,1) (4) {$c$};
    \node at (2,-1) (5) {$d$};
    \node at (-1,0) (6) {};
    \node at (-.05,0) (7) {};
    \draw (1) -- (-.75,0) -- (-.05,0);
    \draw (2) -- (-.75,0);
    \draw (3) -- (-.75,0);
   \draw (4) -- (.75,0) -- (.05,0);
    \draw (5) -- (.75,0);
    \draw (-.45,.2) node {$+$};
    \draw (.45,.2) node {$-$};
    \end{tikzpicture}
\end{center}
\caption{The two possible cases for the ``cut'' graph $\Gamma$ given by $K$}\label{fig:cutgraphs}
\end{figure}
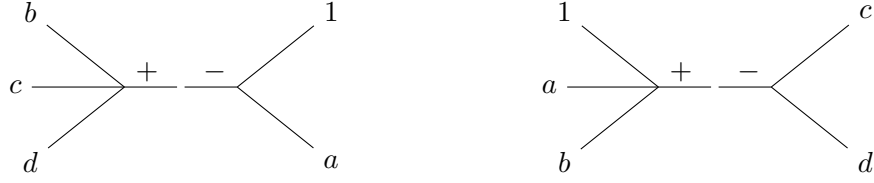

We note here that the stack $\overline{\mathcal{W}}_{0,k}$ has cohomology classes $\tilde \psi_i$, $\tilde \kappa_a$ and $\tilde \Delta_\Gamma$ defined in the same manner, which also satisfy 
\begin{equation}\label{pullbackclasses}
\tilde \psi_i = st^* \psi_i, \quad 
\tilde \kappa_a = st^* \kappa_a, \quad 
r \tilde \Delta_\Gamma = st^* \Delta_\Gamma.
\end{equation}

Recall that the Bernoulli polynomials are a family of polynomials with generating function 
\[
\frac{te^{xt}}{e^t-1} = \sum_{n=0}^\infty B_n(x) \frac{t^n}{n!}.
\]

Hereafter, we use the notational convention that for 
$$\gamma = (\exp{2\pi i \gamma_1},\ldots, \exp{2\pi i \gamma_n}), \quad \theta_i(\gamma) = \gamma_i.$$

In \cite{Chiodo}, Chiodo presented the following formula for the total chern character of the direct image of $R^\bullet \pi_\ast \mathcal{L}$ via the universal curve $\mathcal{C} \overset \pi \to  \mathscr{W}_{g,n}$.  
\begin{align*}
ch_t(R^\bullet \pi_\ast(\mathcal{L}_j ))
= \sum_{k\ge 0} \frac{t^k}{(k+1)!} \Bigg(&B_{k+1}(q_j)\tilde \kappa_k-
\sum_{i=1}^5 B_{k+1}(\gamma^i_j)\tilde \psi_i^k\\
&+\sum_K rB_{k+1}((\gamma_{+}^K)_j) (\rho_K)_\ast\Big(\sum_{i=0}^{k-1}(-\tilde \psi_+)^i(\tilde\psi_-)^{k-i-1}\Big)\Bigg),
\end{align*}
where $B_i$ is the $i$th Bernoulli polynomial,  $q_j = \theta_j(J)$, $\gamma^i_j = \theta_j(\gamma_i)$, and $(\gamma_+^K)_j = \theta_j(\gamma_+^K)$.

Using Equation~\ref{pullbackclasses} we obtain the following version of Chiodo's formula. 
\begin{align}\label{eq:Chiodo}
ch_t(R^\bullet \pi_\ast(\mathcal{L}_j ))
=\sum_{k\ge 0} \frac{t^k}{(k+1)!} st^\ast \Bigg(&B_{k+1}(q_j) \kappa_k-
\sum_{i=1}^5 B_{k+1}(\gamma^i_j)\psi_i^k\nonumber\\
&+\sum_K B_{k+1}((\gamma_{+}^K)_j) (\rho_K)_\ast\Big(\sum_{i=0}^{k-1}(- \psi_+)^i(\psi_-)^{k-i-1}\Big)\Bigg).
\end{align}

Then, 
\begin{align*}
&ch_1(R^\bullet \pi_\ast(\mathcal{L}_j ))=  \frac{1}{2} \Big(B_{2}(q_j)\kappa_1-
\sum_{i=1}^5 B_{2}(\gamma^i_j)\psi_i+\sum_K B_{2}((\gamma_{+}^K)_j) (\rho_K)_\ast(1)\Big),\\
&ch_2(R^\bullet \pi_\ast(\mathcal{L}_j ))=  \frac{1}{6} \Big(B_{3}(q_j)\kappa_2-
\sum_{i=1}^5 B_{3}(\gamma^i_j)\psi_i^2+\sum_K B_{3}((\gamma_{+}^K)_j) (\rho_K)_\ast(\psi_- - \psi_+)\Big).
\end{align*}

%Note that $ch_1^y$ and $ch_2^y$ can be found analogously. 
 The following Lemma can be found in \cite{FFJMR}. 
 \begin{lemma} \label{lem:boundary} Suppose that $|K|>2$, and fix $r,s \in K \setminus \{1\}$. Then
\[
{\rho_{K}}_\ast (\psi_+) = 
\displaystyle\sum_{1,r,s \in I \subset K} \hspace*{-.25cm} \Delta_K \Delta_I
+\hspace*{-.35cm}\displaystyle \sum_{1 \in I \subset K \setminus \{r,s\}} \hspace*{-.25cm}\Delta_K \Delta_{I \cup K^c}.
\] If $|K|\leq 2$, then ${\rho_{K}}_\ast (\psi_+)=0$.

Suppose that $|K|< k-2$, and fix $t,u \in \{1, \ldots, k\}\setminus K$. Then, 
\[
{\rho_{K}}_\ast (\psi_-) =
\displaystyle\sum_{\substack{\emptyset \ne I \subset K^c \\ t,u \notin I }} \Delta_K \Delta_{I\cup K}.
\]
If $|K|\geq k-2$, then ${\rho_{K}}_\ast (\psi_-)=0$.
\end{lemma}
 The following corollary in our case is straightforward to verify. 

\begin{corollary}
We have
\begin{align*}
&(\rho_{\{1,a\}})_\ast(\psi_- - \psi_+) = \Delta_{1,a}\Delta_{1,a,d},
\\
&(\rho_{\{1,a,b\}})_\ast(\psi_- - \psi_+) =-
\Delta_{1,a,b}\Delta_{1,a,b}-\Delta_{1,a,b}\Delta_{1,c,d}.
\end{align*}
\end{corollary}

We will be pushing forward various combinations of these $\psi$, $\kappa$, and $\Delta$ classes. Here we summarize the relevant intersection numbers computed using \cite{Johnson} (See \cite{Faber}), which will allow us to compute our correlator:

\begin{lemma}\label{lem:intnums}
The intersection numbers for various degree 2 classes are given below %in the following table.  
\begin{center}
\textbf{\begin{tabular}{|r|l|}
\hline
class & intersection number \\ \hline \hline
$\kappa_2$&$1$ \\ \hline
$\kappa_1^2$&$5$ \\ \hline
$\kappa_1\psi_i $ &$3$\\ \hline
$\kappa_1\Delta_K$ &$1$ \\ \hline
$\psi_i\psi_j$ & $2-\delta_i^j$\\\hline
$\psi_i\Delta_K$ &$\delta_i^K$\\ \hline
$\Delta_{K}\Delta_{K'}$ &$F(K,K')$\\\hline
%$ \kappa_2$ & $1$\\\hline
%$\psi_i^2$ & $1$\\\hline
%$\Delta_{1,a}\Delta_{1,a,d}$& $1$\\\hline
% $\Delta_{1,c,d}^2 + \Delta_{1,c,d}\Delta_{1,a,b}$& $ -1 + 1$\\\hline
\end{tabular}}
\end{center}

Above, $\delta_i^K$ is 1 when $i$ is in the larger component of the cut graph defined by $ K$, and 0 otherwise;   $F(K, K')$ is given below: 
\[
F(\Gamma_K,\Gamma_I) = 
\left\{ 
\begin{array}{rl}
1 & \text{if } \{K,K'\} = \{\{1,a\},\{1,a,b\}\} \\ 
1 & \text{if } \{K,K'\} = \{\{1,a,b\},\{1,c,d\}\} \\ 
-1 & \text{if } K=K' \\ 
0 & o/w .\\ 
\end{array}
\right.
\]
 \end{lemma}
 
Now, we collect all these facts into the following useful proposition:
  
  \begin{proposition}\label{prop:05corr}
%  Assume that $W \in \Bbb C [x_1,\ldots, x_n]$, and $G$ is an admissible symmetry group for $W$, and that 
Let $\big\langle\,  \fj{1}{\gamma_1}, \fj{1}{\gamma_2}, \fj{1}{\gamma_3}, \fj{1}{\gamma_4}, \fj{1}{\gamma_5}\, \big\rangle_{0, 5}$ be a concave genus-$0$ $5$-point correlator of an admissible pair $(W, G)$ that satisfies \eqref{selection-rule} and \eqref{virtual-cycle-degree}. There are two cases:
%\textcolor{red}{which satisfies the non-vanishing criteria in Lemma 3.10}\footnote{Where do we need this condition?}.
\begin{enumerate}
    \item 
If ${\rm rank}\, \pi_*\mathcal{L}_{i_1} = {\rm rank}\, \pi_*\mathcal{L}_{i_2}=1$, while the ranks of all other pushforward bundles are equal to 0, then 
\begin{align*}
& \big\langle\,  \fj{1}{\gamma_1}, \fj{1}{\gamma_2}, \fj{1}{\gamma_3}, \fj{1}{\gamma_4}, \fj{1}{\gamma_5}\, \big\rangle_{0,5}\\
=&{1\over 4}\Big(5{B}(J,J) +\sum_{i=1}^5 ({B}(\gamma_i,\gamma_i)-3{\tilde B}(J,\gamma_i))
  +\sum_K ({\tilde B}(J, \gamma_{+}^K)-B(\gamma_+^K,\gamma_+^K))+2\!\!\!\sum_{1\leq i<j\leq5}\!\!\!\tilde B(\gamma_i,\gamma_j)\\
&  +\!\!\!\sum_{1< i< j\le 5}\!\! \big(
  -\tilde B(\gamma_1,\gamma_{+}^{1,i,j})
 - \tilde B(\gamma_i, \gamma_+^{1,i,j})-\tilde B(\gamma_i, \gamma_+^{\{1,j\}^c})+\tilde B(\gamma_+^{1,i},\gamma_+^{1,i,j}) + \tilde B(\gamma_+^{1,i,j},\gamma_+^{\{i,j\}^c})
 \big)\Big),
\end{align*}
where $B(\gamma_1,\gamma_2) = B_2(\theta_{i_1}(\gamma_1))B_2(\theta_{i_2}(\gamma_2))$ and $\tilde B(\gamma_1,\gamma_2) = B(\gamma_1,\gamma_2)+B(\gamma_2,\gamma_1)$.\\
   \item 
If ${\rm rank}\, \pi_*\mathcal{L}_{a} =2$, while the ranks of all other pushforward bundles are equal to 0, then
%Then the correlator $\big\langle\,  \fj{1}{\gamma_1}, \fj{1}{\gamma_2}, \fj{1}{\gamma_3}, \fj{1}{\gamma_4}, \fj{1}{\gamma_5}\, \big\rangle_{0,5}$ is equal to 
\begin{align*}
&\big\langle\,  \fj{1}{\gamma_1}, \fj{1}{\gamma_2}, \fj{1}{\gamma_3}, \fj{1}{\gamma_4}, \fj{1}{\gamma_5}\, \big\rangle_{0,5}\\
=&\frac{1}{8} \Big ( 5{B}'_2(J,J)
 +\sum_{i=1}^5 ({B}'_2(\gamma_i,\gamma_i)-6{ B}'_2(J,\gamma_i))
  +\sum_K ({2 B}'_2(J, \gamma_{+}^K)-B(\gamma_+^K,\gamma_+^K))+4\!\!\!\!\!\sum_{1\leq i<j\leq5}\!\!\!\! B'_2(\gamma_i,\gamma_j)\\
& +2\!\!\!\sum_{1< i< j\le 5}\!\! \big(
  - B'_2(\gamma_1,\gamma_{+}^{1,i,j})
 -  B'_2(\gamma_i, \gamma_+^{1,i,j})- B'_2(\gamma_i, \gamma_+^{\{1,j\}^c})+ B'_2(\gamma_+^{1,i},\gamma_+^{1,i,j}) +  B'_2(\gamma_+^{1,i,j},\gamma_+^{\{i,j\}^c})\big)
   \Big)\\
& +\frac{1}{6}  \Big({B}'_3(J)-
\sum_{i=1}^5 {B}'_3(\gamma_i)+
\sum_{i=1}^5 {B}'_3(\gamma_{+}^{1,i})\Big).
\end{align*}
\end{enumerate}
where $B'_3(\gamma) = B_3(\theta_{i_0}(\gamma))$, and $B'_2(\gamma_1,\gamma_2) =B_2(\theta_{i_0}(\gamma_1))B_2(\theta_{i_0}(\gamma_2)) $.
\end{proposition}

\subsection{A proof of Proposition \ref{nonvanishing-appendix}}
\iffalse
Let $W =x^a+y^b$, and $G=\langle J \rangle$.
In this section we will compute the correlator $\langle\,  \fj{1}{J^2}, \fj{1}{J^2}, \fj{1}{J^2}, \fj{1}{J^{d-2}}, \fj{1}{J^{d-1}}\, \rangle_{0,5}^{W,G}$. 
Set $d ={\rm lcm}(a,b) $, $a = d/w_1, b = d/w_2$.
\fi

Now we evaluate the correlator $\<\alpha_{d-1},  \alpha_{d-2},  \alpha_{2},  \alpha_{2},  \alpha_{2}\>_{0,5}$ in Proposition~\ref{nonvanishing-appendix}.
We use Proposition~\ref{prop:05corr}, as this correlator satisfies the necessary hypotheses.
We define
\begin{equation}
\label{bernoulli-formula}
\begin{dcases}
b_k^x: =B_2\left({k w_1\over d}\right)=\frac{d^2-6d k w_1+6k^2w_1^2}{6d^2}, \\
b_k^y: =B_2\left({k w_2\over d}\right)=\frac{d^2-6d k w_2+6k^2w_2^2}{6d^2}.
\end{dcases}
\end{equation}
We denote by $b_0:=b_0^x=b_0^y.$
%Recall that \begin{equation}\label{bernoulli-formula}
%\begin{dcases}
%B_2(x)  =B_2(1-x) = x^2-x+\frac 16.
%B_3(x) =-B_3(1-x) =  x^3-\frac 32 x^2 + \frac 12 x.
%\end{dcases}
%\end{equation}
The following lemma will help with our calculations. Here, for brevity, we set 
$$\gamma^i_x = \theta_x(\gamma_i), \quad q_x = \theta_x(J),\quad \text{and}\quad \gamma_{+,x}^{K} = \theta_x(\gamma_{+}^{K}).$$
\begin{lemma}\label{lem:bernoulli}
We have %\textcolor{red}{We have $q_i\neq {1\over 2}$, $\frac {2w_1}{d}<1$, what about $\frac {3w_1}{d}, \frac {4w_1}{d}?$}
\begin{align*}
   & {B}_2(\gamma_{+,x}^{\{1,5\}}) = {B}_2(\gamma_{+,x}^{\{1,2,4\}}) = {B}_2(\gamma_{+,x}^{\{1,3,4\}}) =b_0, \\ %=  \frac16
&   {B}_2(q_x) ={B}_2(\gamma_x^5)  = {B}_2(\gamma_{+,x}^{\{1,4\}}) = {B}_2(\gamma_{+,x}^{\{1,2,5\}})={B}_2(\gamma_{+,x}^{\{1,3,5\}})=b_1^x, %= \frac{d^2-6dw_1+6w_1^2}{6d^2}
\\
 &  {B}_2(\gamma_x^1) ={B}_2(\gamma_x^2) = {B}_2(\gamma_x^3)={B}_2(\gamma_x^4) =b_2^x, %  = \frac{d^2-12dw_1+24w_1^2}{6d^2}
\\
 &  {B}_2(\gamma_{+,x}^{\{1,4,5\}}) ={B}_2\Big(\gamma_{+,x}^{\{1,2\}}\Big) ={B}_2(\gamma_{+,x}^{\{1,3\}})=b_3^x, %= \frac{d^2-18dw_1+54w_1^2}{6d^2}
\\
  & {B}_2(\gamma_+^{\{1,2,3\}})=
  \begin{dcases}
  b_4^x, & \text{if} \quad 4 w_1\leq d; \\
   \widetilde{b}_4^x, &  \text{if} \quad 4 w_1>d,
  % = \frac{d^2-24dw_1+96w_1^2}{6d^2}
  \end{dcases}
   \end{align*}
where $ \widetilde{b}_4^x = B_2\left({k w_1\over d}-1\right)$.

Equivalent statements hold for $q_y$, $\gamma^i_y$, and  $\gamma^K_y$ with $w_2$ in place of $w_1$. 
\end{lemma}
\begin{proof} 
Let $\{r\}$ be the fractional part of $r\in \mathbb{R}$.
It is straightforward to verify that 
\begin{align*}
&\gamma_+^{\{1,i\}} = J - \gamma_1 - \gamma_i
= \left\{ 
\begin{array}{rl}
\{1-\frac{3w_1}{d} \} & \text{if } i \in \{2,3\},\\[1mm]
\frac{w_1}{d}  & \text{if } i =4,\\[1mm]
0  & \text{if } i =5,\\
\end{array}
\right. 
\end{align*}
and 
\begin{align*}
&\gamma_+^{\{1,i,j\}} = 2J - \gamma_1 - \gamma_i - \gamma_j
= \left\{ 
\begin{array}{rl}
\{1-\frac{4 w_1}{d}\} & \text{if } \{i,j\} = \{2,3\},\\[1mm]
0  & \text{if } \{i,j\} = \{2,4\} \text{ or } \{3,4\},\\[1mm]
1-\frac{w_1}{d}  & \text{if } \{i,j\} = \{2,5\} \text{ or } \{3,5\},\\[1mm]
\{\frac{3w_1}{d}\}  & \text{if } \{i,j\} = \{4,5\}.\\
\end{array}
\right.  
\end{align*}
In Proposition~\ref{nonvanishing-appendix}, $W:\C^2\to \C$ is assumed to be a degree $d$ invertible polynomial with $\widehat{c}_W\geq 1$.
So we have $${kw_1\over d}, {kw_2\over d}\in [0, 1], \quad \forall k=0, 1, 2, 3,$$ 
and at most one of ${4w_1\over d}, {4w_2\over d}$ belongs to the interval $(1, 2)$.
Now the result follows from~\eqref{bernoulli-formula} and the equality $B_2(x)  =B_2(1-x).$ %= x^2-x+\frac 16$.
\end{proof}

\iffalse
\begin{proposition}
 Assume that $W =x^d+y^d$, and $G=\langle J \rangle$.
Then
\[
\big\langle\,  \fj{1}{J^2}, \fj{1}{J^2}, \fj{1}{J^2}, \fj{1}{J^{d-2}}, \fj{1}{J^{d-1}}\, \big\rangle_{0,5} = \frac {w_1w_2} {d^2}
\]
\end{proposition}
\fi

%\newpage
\noindent{\em A proof of Proposition~\ref{nonvanishing-appendix}.}
Note that in this case the  pushforward bundles $\pi_*( \mathcal{L}_x)$ and $\pi_*( \mathcal{L}_y)$ are both of rank one. 
We discuss in two cases:
\begin{enumerate}
\item ${4w_1\over d}\leq 1$ and ${4w_2\over d}\leq 1$;
\item one of ${4w_1\over d}, {4w_2\over d}$ belongs to $(1, 2)$.
\end{enumerate}
For the first case, we have
\begin{align*}
&c_2(R^1\pi_\ast( \bigoplus_j \mathcal{L}_j) \\
=\frac{1}{4} & \Big( b_1^x(\kappa_1-\psi_5+\Delta_{1,4}+\Delta_{1,2,5}+\Delta_{1,3,5})
 +b_2^x(-\psi_1-\psi_2-\psi_3-\psi_4)\\
&+b_3^x(\Delta_{1,2}+\Delta_{1,3}+\Delta_{1,4,5})+b_4^x\Delta_{1,2,3}+b_0(\Delta_{1,5}+\Delta_{1,2,4}+\Delta_{1,3,4})\Big)\times\\
&\Big( b_1^y(\kappa_1-\psi_5+\Delta_{1,4}+\Delta_{1,2,5}+\Delta_{1,3,5})
 +b_2^y(-\psi_1-\psi_2-\psi_3-\psi_4)\\
& +b_3^y(\Delta_{1,2}+\Delta_{1,3}+\Delta_{1,4,5})+b_4^y\Delta_{1,2,3}+b_0(\Delta_{1,5}+\Delta_{1,2,4}+\Delta_{1,3,4})\Big).
\end{align*}

Consider the expression $4c_2(R^1\pi_\ast( \bigoplus_j \mathcal{L}_j)$ as a polynomial in the variables $b_i^j$ with coefficients which are expressions involving our special cohomology classes. Each coefficient can then be pushed forward to obtain an intersection number (see Lemma \ref{lem:intnums}). The intersection number coefficients are summarized below. 

%\newpage

\begin{table}[H]
\begin{center}
    \begin{tabular}{|c|r|}
    \hline
    monomial $m$ & intersection number $n$\\ \hline \hline
    $b_0^2$ & $-1-1-1=-3$\\\hline
    $b_1^xb_1^y$ & $5-5-1-1-1=-3$\\\hline
    $b_2^xb_2^y$ & $4(7)=28$\\\hline
    $b_3^x b_3^y$ & $-1-1-1=-3$\\\hline
    $b_4^xb_4^y$ & $-1$\\\hline
    $b_0b_1^x$, $b_0b_1^y$ & $3+0+2+2+2=9$\\\hline
    $b_0b_2^x$, $b_0b_2^y$ & $-2-2-2-3=-9$\\\hline
    $b_0b_3^x$, $b_0b_3^y$ & $1+1+1=3$\\\hline
    $b_0b_4^x$, $b_0b_4^y$ & $0$\\\hline
    $b_1^xb_2^y$, $b_1^yb_2^x$ & $-12+8-2-2-2=-10$\\\hline
    $b_1^xb_3^y$, $b_1^yb_3^x$ & $3-3+1+1+1=3$\\\hline
    $b_1^xb_4^y$, $b_1^yb_4^x$ & $1$\\\hline
    $b_2^xb_3^y$, $b_2^yb_3^x$ & $-3-1-1-1=-6$\\\hline
    $b_2^xb_4^y$, $b_2^yb_4^x$ & $-3$\\\hline
    $b_3^xb_4^y$, $b_3^yb_4^x$ & $3$\\\hline
    \end{tabular}
\end{center}
\end{table}

%\newpage
Next, we use Lemma \ref{lem:bernoulli} to express the terms $b_ib_j$ as functions of $d$, which can then be multiplied by the intersection number coefficients found in the table here.

\begin{table}[H]
 \resizebox{\textwidth}{!}{%
%\begin{center}
%\begin{adjustbox}{max width=\linewidth,center}
    \begin{tabular}{|c|l|l|}
    \hline
     $m$  &  $36d^4\cdot m\cdot n$ \\\hline \hline
    $b_0^2$ & $-3d^4$ %& $-3d^4$
    \\\hline
    $b_1^xb_1^y$ & $-3d^4 +18d^3(w_1+w_2)-18 d^2(w_1^2+w_2^2+6w_1w_2) \  +108d(w_1^2w_2 + w_1w_2^2) \ -108w_1^2w_2^2$
    \\\hline
    $b_2^xb_2^y$ & $28d^4 - 336 d^3(w_1+w_2) + 672 d^2(w_1^2+w_2^2+6w_1w_2) - 8064 d(w_1^2w_2 + w_1w_2^2) + 16128w_1^2w_2^2$
    \\\hline
    $b_3^xb_3^y$ & $-3d^4 +54 d^3(w_1+w_2)  -162 d^2(w_1^2+w_2^2+6w_1w_2) +2916 d(w_1^2w_2 + w_1w_2^2) -8748 w_1^2w_2^2$
    \\\hline
    $b_4^xb_4^y$ & $-d^4 \ + 24d^3(w_1+w_2) - 96d^2(w_1^2+w_2^2+6w_1w_2) + 2304d(w_1^2w_2 +w_1w_2^2) - 9216w_1^2w_2^2$
    \\\hline
    $b_0b_1^x$ & $\ \ 9d^4 - 54 d^3w_1 \qquad \qquad+ 54d^2w_1^2$
    \\\hline
    $b_0b_1^y$ & $\ \ 9d^4\qquad\qquad - 54 d^3w_2 \qquad \qquad+ 54d^2w_2^2$
    \\\hline
    $b_0b_2^x$ & $-9d^4 +108 d^3w_1 \qquad  -216 d^2w_1^2$
    \\\hline
    $b_0b_2^y$ & $-9 d^4 \qquad \qquad +108 d^3w_2 \qquad  -216 d^2w_2^2$
    \\\hline
    $b_0b_3^x$ & $\ \ 3 d^4 - 54 d^3w_1 \qquad \qquad+ 162 d^2w_1^2$
    \\\hline
    $b_0b_3^y$ & $\ \ 3 d^4 \qquad \qquad- 54 d^3w_2 \qquad \qquad+162d^2w_2^2$
    \\\hline
    $b_1^xb_2^y$ & $-10d^4 + 60 d^3(w_1+2w_2)- 60d^2(w_1^2+4w_2 ^2+12w_1w_2) +720d(w_1^2w_d+2w_1w_2^2) - 1440w_1^2w_2^2$
    \\\hline
    $b_1^yb_2^x$ & $-10d^4 + 60d^3(2w_1+w_2)- 60d^2(4w_1^2+w_2 ^2+12w_1w_2) +720d(2w_1^2w_d+w_1w_2^2) - 1440w_1^2w_2^2$
    \\\hline
    $b_1^xb_3^y$ & $\ \ 3d^4 \ - 18d^3(w_1 + 3w_2) \ + 18 d^2(9w_1^2 + w_2^2 + 18w_1w_2) - 324 d(w_1^2w_2+3w_1w_2^2) + 972 w_1^2w_2^2$
    \\\hline
   $b_1^yb_3^x$ & $\ \ 3d^4 \ - 18 d^3(3w_1 + w_2) \ + 18 d^2(w_1^2 + 9w_2^2 + 18w_1w_2) - 324 d(3w_1^2w_2+w_1w_2^2) + 972 w_1^2w_2^2$
    \\\hline
    $b_1^xb_4^y$ & $\ \ \ d^4 \ \ - 6d^3(w_1+4w_2)\ \ + 6d^2(w_1^2+16w_2^2+24w_1w_2) - 144d(w_1^2w_2 + 4w_1w_2^2) + 576w_1^2w_2^2$
    \\\hline
    $b_1^yb_4^x$ & $\ \ \ d^4 \ \ - 6d^3(4w_1+w_2) \ \ + 6d^2(16w_1^2+w_2^2+24w_1w_2) - 144d(4w_1^2w_2 + w_1w_2^2) + 576w_1^2w_2^2$
    \\\hline
    $b_2^xb_3^y$ & $-6d^4 +36 d^3(2w_1 + 3w_2)-36 d^2(4w_1^2 + 9w_2^2 + 36w_1w_2) +1296 d( 2w_1^2w_2+3w_1w_2^2 ) -7776 w_1^2w_2^2$
    \\\hline
    $b_2^yb_3^x$ & $-6d^4 + 36d^3(3w_1 + 2w_2) - 36d^2(9w_1^2 + 4w_2^2 + 36w_1w_2) 
    +1296d(3w_1^2w_2+2w_1w_2^2 ) - 7776w_1^2w_2^2$
    \\\hline
    $b_2^xb_4^y$ & $-3d^4 +36 d^3(w_1+2w_2) \ -72 d^2(w_1^2+4w_2^2+12w_1w_2) + 1728 d(w_1^2w_2+2w_1w_2^2)  -6912 w_1^2w_2^2$
    \\\hline
    $b_2^yb_4^x$ & $-3d^4 +36 d^3(2w_1+w_2) \ - 72 d^2(4w_1^2+w_2^2+12w_1w_2) +1728 d(2w_1^2w_2+w_1w_2^2) -6912w_1^2w_2^2$
    \\\hline
    $b_3^xb_4^y$ & $3d^4 - 18 d^3(3w_1+4w_2) + 18 d^2(9w_1^2+16w_2^2+72w_1w_2) -1296 d(3w_1^2w_2 + 4w_1w_2^2) + 15552 w_1^2w_2^2$
    \\\hline
    $b_3^yb_4^x$ & $3d^4 - 18d^3(4w_1+3w_2) + 18 d^2(16w_1^2+9w_2^2+72w_1w_2) -1296d(4w_1^2w_2 + 3w_1w_2^2) + 15552w_1^2w_2^2$
    \\\hline\hline
    &{Total}: $144d^2w_1w_2$\\\hline
    \end{tabular}
    }
%\end{center}
%\end{adjustbox}
\end{table}

Dividing the total in the table above by $4 \cdot 36d^4$ gives the desired result.

The second case,  when ${4w_1\over d}\in (1,2)$,  can be obtained similarly.
According to Lemma~\ref{lem:bernoulli}, we just need to replace $b_4^x$ by $\widetilde{b}_4^x$ in $c_2(R^1\pi_\ast( \bigoplus_j \mathcal{L}_j)$.
The difference is 
$$(\widetilde{b}_4^x-b_4^x)\left(-b_4^y+b_1^y-3b_2^y+3b_3^y\right).$$
Using \eqref{bernoulli-formula}, we have $-b_4^y+b_1^y-3b_2^y+3b_3^y=0.$ 
So the difference is zero and the result follows.
\qed

%\newpage

\vskip .2in

\noindent{\small Mathematical Reviews, American Mathematical Society,   Ann Arbor, Michigan 48103, USA}

\noindent{\small E-mail: aef@ams.org}

\vskip .1in

\noindent{\small Department of Mathematics, Sun Yat-sen University, Guangzhou, Guangdong 510275,
China}

\noindent{\small E-mail: hewq@mail2.sysu.edu.cn}

\vskip .1in

\noindent{\small Department of Mathematics, University of Oregon, Eugene, OR 97403,
USA}

\noindent{\small E-mail: yfshen@uoregon.edu}


\begin{thebibliography}{99}
\bibitem{Abr}
Lowell Abrams,
{\em The quantum Euler class and the quantum cohomology of the Grassmannians,}
Israel J. Math. 117 (2000) 335--352.

\bibitem{Acosta}
Pedro Acosta,
{\em FJRW-Rings and Landau-Ginzburg Mirror Symmetry in Two Dimensions,}
arXiv:0906.0970 [math.AG]

%\bibitem{AT} Daisuke Aramaki,  Atsushi Takahashi, {\em Maximally-graded matrix factorizations for an invertible polynomial of chain type.} Adv. Math. 373 (2020), 107320, 23 pp.

\bibitem{Bayer}
Arend Bayer, 
{\em Semisimple quantum cohomology and blowups.}
Int. Math. Res. Not. 2004, no. 40, 2069--2083.


\bibitem{BM}
Arend Bayer, Yuri I. Manin,
{\em (Semi)simple exercises in quantum cohomology.}
The Fano Conference, 143--173, Univ. Torino, Turin, 2004. 

\bibitem{BP}
 Alexey Bondal, Alexander Polishchuk, 
{\em Homological properties of associative algebras: the method of helices.}
%Izv. Ross. Akad. Nauk Ser. Mat. 57 (1993), no. 2, 3–50; translation in
Russian Acad. Sci. Izv. Math. 42 (1994), no. 2, 219--260

\bibitem{Chiodo}
Alessandro Chiodo,
{\em Towards an enumerative geometry of the moduli space of twisted curves and rth roots.}
Compos. Math., 144(6):1461--1496, 2008.

%\bibitem{Dyc} T Dyckerhoff,  {\em Compact generators in categories of matrix factorizations.} Duke Math. J. 159 (2011), no. 2, 223--274.


\bibitem{Dubrovin}
Boris Dubrovin, 
{\em Geometry and analytic theory of Frobenius manifolds,}
 Doc. Math. (1998), no. Extra Vol. II, 315--326.

\bibitem{Faber}
Carel Faber,
{\em Algorithms for computing intersection numbers on moduli spaces of curves, with an application to the class of the locus of Jacobians.}
In New trends in algebraic geometry (Warwick, 1996), volume 264 of London Math. Soc. Lecture Note Ser., pages 93--109. Cambridge Univ. Press, Cambridge, 1999.


\bibitem{FFJMR}
Huijun Fan, Amanda Francis, Tyler J. Jarvis, Evan Merrell, Yongbin Ruan,
{\em Witten's $D_4$ integrable hierarchies conjecture.}
Chin. Ann. Math. Ser. B 37 (2016), no. 2, 175--192.


\bibitem{FJR1}
Huijun Fan, Tyler Jarvis, Yongbin Ruan,
{\em The Witten equation and its virtual fundamental cycle.}
 arXiv:0712.4025 (2007).

\bibitem{FJR}
Huijun Fan, Tyler Jarvis, Yongbin Ruan,
{\em The Witten equation, mirror symmetry and quantum singularity theory.}
 Ann. of Math. (2) 178 (2013), no. 1, 1--106.

\bibitem{FKK}
David Favero, Daniel Kaplan, Tyler L. Kelly,
{\em Exceptional Collections for Mirrors of Invertible Polynomials.}
arXiv:2001.06500 [math.AG]


\bibitem{Francis}
Amanda Francis,
{\em Computational techniques in FJRW theory with applications to Landau-Ginzburg mirror symmetry.}
 Adv. Theor. Math. Phys. 19 (2015), no. 6, 1339--1383.


\bibitem{G-ss}
Alexander B. Givental,
{\em Semisimple Frobenius structures at higher genus.} Internat. Math. Res. Notices 2001, no. 23, 1265--1286.
 
 \bibitem{G-vir}
Alexander B. Givental, 
{\em Gromov-Witten invariants and quantization of quadratic Hamiltonians.}
 Dedicated to the memory of I. G. Petrovskii on the occasion of his 100th anniversary. Mosc. Math. J. 1 (2001), no. 4, 551--568, 645.
 
 
 \bibitem{Habermann-hms}
Matthew Habermann,
{\em  Homological mirror symmetry for invertible polynomials in two variables.} Quantum Topol. 13 (2022), no. 2, pp. 207--253.
 
\bibitem{Habermann-curve}
Matthew Habermann,
{\em A note on homological Berglund-H\"ubsch-Henningson mirror symmetry for curve singularities.} arXiv:2205.12947 [math.SG]
 
 
 \bibitem{HaSm}
 Matthew Habermann,Jack Smith,
 {\em Homological Berglund-H\"ubsch mirror symmetry for curve singularities,} J. Symplectic Geom. Volume 18 (2020), no. 6, 1515--1574. 
 
 
\bibitem{HLSW}
Weiqiang He, Si Li, Yefeng Shen, Rachel Webb,
{\em Landau-Ginzburg mirror symmetry conjecture,}
J. Eur. Math. Soc. 24 (2022), no. 8, pp. 2915--2978.

%\bibitem{HPSV} Weiqiang He, Alexander Polishchuk, Yefeng Shen, Arkady Vaintrob, {\em A Landau-Ginzburg mirror theorem via matrix factorizations,} J. Reine Angew. Math. 794 (2023), 55--100.



\bibitem{HS}
Weiqiang He, Yefeng Shen,
{\em Virasoro constraints in quantum singularity theories,}
arXiv:2103.00313 [math.AG]

\bibitem{HO}
Yuki Hirano, Genki Ouchi, 
{\em Derived factorization categories of non-Thom--Sebastiani-type sums of potentials,}
arXiv:1809.09940 [math.AG]
To appear in Proceedings of the London Mathematical Society



\bibitem{Johnson}
Drew Johnson,
{\em Top Intersections on $\overline{M}_{g,n}$,} 
August 2011. https://bitbucket.org/drew j/top- intersections-on-mbar g-n.

%\bibitem{Kap} M. M.  Kapranov,  {\em Derived category of coherent bundles on a quadric.}   Funktsional. Anal. i Prilozhen. 20 (1986), no. 2, 67. 
 

\bibitem{Ke}
Hua-Zhong Ke, 
{\em On semisimplicity of quantum cohomology of $\mathbb{P}^1$-orbifolds.}
 J. Geom. Phys. 144 (2019), 1--14.

%\bibitem{Kravets} Oleksandr Kravets, {\em Categories of singularities of invertible polynomials,} arXiv:1911.09859 [math.AG]

\bibitem{KS}
Maximilian Kreuzer,  Harald Skarke,
{\em On the classification of quasihomogeneous functions.}
 Comm. Math. Phys. 150 (1992), no. 1, 137--147.

\bibitem{LSZ}
Jun Li, Yefeng Shen, Jie Zhou, 
{\em Higher Genus FJRW Invariants of a Fermat Cubic,} 
arXiv:2001.00343. 
To appear in Geometry $\&$ Topology.

\bibitem{Sai}
Kyoji Saito, 
{\em Period mapping associated to a primitive form.}
Publ. Res. Inst. Math. Sci. 19 (1983), no. 3, 1231--1264.


\bibitem{Teleman}
Constantin Teleman, 
{\em The structure of 2D semi-simple field theories,} 
Invent. Math. 188 (2012), no. 3, 525--588.

\end{thebibliography}
\end{document}